\documentclass[reqno,11pt,a4paper]{amsart} 
\usepackage{csquotes} 
\usepackage{comment} 
\usepackage{mathtools}
\usepackage{amssymb}
\usepackage{mleftright}  
\usepackage{xfrac}
\usepackage{mathrsfs}
\usepackage[dvipsnames]{xcolor}
\usepackage[colorlinks,linkcolor=RedViolet,citecolor=PineGreen,anchorcolor=Green,urlcolor=RedViolet]{hyperref}
\usepackage[headings]{fullpage}
\usepackage[numbers,square]{natbib}
\usepackage[shortlabels,inline]{enumitem}
\usepackage{anyfontsize}
\usepackage{scalerel}
\usepackage[foot]{amsaddr}
\usepackage{bbm}

\makeatletter
\newcommand*\rel@kern[1]{\kern#1\dimexpr\macc@kerna}
\newcommand*\widebar[1]{%
  \begingroup
  \def\mathaccent##1##2{%
    \rel@kern{0.8}%
    \overline{\rel@kern{-0.8}\macc@nucleus\rel@kern{0.2}}%
    \rel@kern{-0.2}%
  }%
  \macc@depth\@ne
  \let\math@bgroup\@empty \let\math@egroup\macc@set@skewchar
  \mathsurround\z@ \frozen@everymath{\mathgroup\macc@group\relax}%
  \macc@set@skewchar\relax
  \let\mathaccentV\macc@nested@a
  \macc@nested@a\relax111{#1}%
  \endgroup
}
\makeatother
  
\newcommand{\E}{\textsf{\upshape E}}
\renewcommand{\P}{\textsf{\upshape P}}
\newcommand{\Q}{\textsf{\upshape Q}}
\newcommand{\indicator}[1]{\mathbf{1}_{#1}}

\DeclareMathOperator{\trace}{trace}

\newcommand{\sigalg}[1]{\mathscr{#1}} 

\newcommand{\Exp}{\mathscr{E}}

\DeclarePairedDelimiter{\abs}{\lvert}{\rvert}

\newcommand{\lc}{[\![}

\newcommand*{\bigs}[1]{\scalebox{1.2}{\ensuremath#1}}

\newcommand{\N}{\mathbb{N}}

\newcommand{\R}{\mathbb{R}}

\DeclareMathOperator{\id}{id}
\DeclareMathOperator{\tr}{tr}

\newcommand{\rt}{\rho}
\newcommand{\lopt}{\widehat\lambda}
\newcommand{\piopt}{\widehat\pi}
\newcommand{\lLRT}{\lambda^{\rt}}

\newcommand{\F}{\sigalg{F}}
\newcommand{\D}{\mathcal{D}}

\newcommand{\xinf}{\underline{x\mkern-1.5mu}\mkern1.5mu}
\newcommand{\xinfLRT}{\xinf_\rt}
\renewcommand{\d}{\mathrm{d}}
\DeclareMathOperator*{\argmax}{arg\,max}

\newcommand{\gfrak}{\mathfrak{g}}
\newcommand{\ffrak}{\mathfrak{f}}

\newcommand{\gLRT}{g_{\rt}}

\newcommand{\e}{\mathrm{e}}
\DeclareMathOperator{\edom}{e-dom}
\DeclareMathOperator{\dom}{dom}

\newcommand\shorturl[1]{%
  \href{http://#1}{\nolinkurl{#1}}%
}

\makeatletter
\newcommand*{\bigcdot}{}
\DeclareRobustCommand*{\bigcdot}{%
  \mathbin{\mathpalette\bigcdot@{}}%
}
\newcommand*{\bigcdot@scalefactor}{.7}
\newcommand*{\bigcdot@widthfactor}{1.15}
\newcommand*{\bigcdot@}[2]{%
  \sbox0{$#1\vcenter{}$}
  \sbox2{$#1\cdot\m@th$}%
  \hbox to \bigcdot@widthfactor\wd2{%
    \hfil
    \raise\ht0\hbox{%
      \scalebox{\bigcdot@scalefactor}{%
        \lower\ht0\hbox{$#1\bullet\m@th$}%
      }%
    }%
    \hfil
  }%
}
\makeatother
\newcommand{\sint}{\bigcdot}

\newtheorem{theorem}{Theorem}[section]
\theoremstyle{plain}

\newtheorem{corollary}[theorem]{Corollary}

\newtheorem{definition}[theorem]{Definition}
\newtheorem{example}[theorem]{Example}

\newtheorem{lemma}[theorem]{Lemma}

\newtheorem{proposition}[theorem]{Proposition}
\newtheorem{remark}[theorem]{Remark}

\numberwithin{equation}{section}

\begin{document}
\title{On integrability of local utility maximizing strategies}
\author{Ale\v{s} \v{C}ern\'{y}$^{a,b}$}
\address{$^a$Bayes Business School, City St George's, University of London\\
106 Bunhill Row, London EC1Y 8TZ, UK\vspace{-0.1ex}}
\email{\!ales.cerny@citystgeorges.ac.uk}
\author{Johannes Ruf$^b$}
\email{j.ruf@lse.ac.uk}
\address{$^b$Department of Mathematics\\
  London School of Economics and Political Science\\
	Houghton Street, London, WC2A 2AE, UK\vspace{-0.1ex}}

\date{\today}
    
\begin{abstract}
In stochastic control problems involving optimal portfolio choice, one often seeks an optimizer of some local criterion (dynamic programming value function) generated by some global criterion (e.g., expected utility). The local optimizer then typically acts as a stochastic integrand with respect to some multivariate stochastic process (e.g., asset prices). This paper takes a closer look at local criteria in the form of ``local expected utility''. Since the existence of the local optimizer can be established under very mild conditions, it is crucial to understand under what circumstances the local optimizer has good integrability properties. It is shown that the reasonable asymptotic elasticity condition of Kramkov and Schachermayer \nocite{kramkov.schachermayer.99} applied to the local utility function plays a key role in this context.
\end{abstract}

\maketitle
\mleftright 

\section{Introduction}
We isolate and study a problem encountered in financial mathematics, whereby one seeks a local (in time) optimizer of some local criterion (here modelled as the local expected utility) with a view that the local optimizer will then be optimal for some global criterion that generated the local value function in the first place. We have in mind applications, where the local optimizer  acts as a stochastic integrand (optimal portfolio weights) $\lambda$ with respect to some underlying exogenous stochastic process (asset returns) $R$. The main aim of this paper is to understand under what circumstances the local optimizer has good integrability properties. 

Local expected utility in this paper is a predictable concave function $\gfrak$ obtained as the drift rate of the $g$-variation of a $d$-dimensional stochastic integral $\int_0^\cdot\lambda_t\d R_t$, for some concave function $g$. The predictable criterion $\gfrak(\lambda)$ is to be maximized pointwise, for each $(\omega,t)$, over $\lambda\in\R^d$. A predictable optimizer, if it exists, is then referred to as the \emph{local optimizer}, say $\lopt$. Our aim is first to establish existence of local optimizers and then formulate necessary and sufficient criteria for $\lopt$ to integrate $R$ in terms of the optimized value $\gfrak(\lopt)$.

This work is not the first to look at local utility. The concept was originally proposed and studied in \citet{kallsen.99,kallsen.02} for a specific class of time-constant deterministic concave functions $g$ without any link to global utility maximization. Related local criteria subsequently appeared in several global utility-maximization problems, often expressed through model-specific objectives or first-order conditions. In the context of (exponential) L\'evy models, local expected utility maximization for specific functions $g$ appears in, e.g., \citet{kallsen.00}, \citet{kardaras.09}, and \citet{nutz.12.mafi}.
It also features indirectly in the form of first-order conditions in \citet{jeanblanc.al.07} and \citet{bender.niethammer.08}. For deterministic time-constant $g$ and L\'evy processes, however, integrability of the local optimizer is not an issue since the optimizer, if it exists, can be chosen deterministic and time-constant.

Applications where time-inhomogeneous maximization of local expected utility appears include \citet{goll.kallsen.00,goll.kallsen.03}    and \citet{karatzas.kardaras.07} (logarithmic utility), \citet{kallsen.muhle-karbe.10.aap,kallsen.muhle-karbe.10.ijtaf} (power utility),  \citet{fujiwara.10} (exponential utility), and  \citet{crs.25.aheadofprint} (monotone quadratic utility). In time-inhomogeneous models, a systematic study of the existence and integrability of local utility maximizers has the potential to simplify existing and uncover new results. Below, we illustrate this point by simplifying the  assumptions needed in \citet[Theorem~3.1]{kallsen.muhle-karbe.10.aap}.
 
The hope of formulating universal integrability theorems for local optimizers comes from a recent  result of \citet*[Theorem~2.1]{crs.25.aheadofprint}, which establishes equivalence between (i) finiteness of cumulative local utility generated by the local optimizer; and (ii) integrability of the local optimizer with respect to the asset price process, for the monotone quadratic utility function 
\begin{equation}\label{eq:gMMV}
    g = \id\wedge1-\frac{1}{2}(\id\wedge1)^2
\end{equation}
and semimartingale integrators with independent increments satisfying a mild no-arbitrage condition.

In this paper we establish a similarly close relationship between (i) and (ii) for a broad subset of (for the time being time-constant and deterministic) non-decreasing concave functions $g$ twice Peano-differentiable at $0$ with $g(0)=0$, and general instantaneously arbitrage-free semimartingale integrators. Corollary~\ref{C} shows that (ii)\!\! $\implies$\!\! (i) if $g$ is bounded above (and the local optimizer exists).  
Theorem~\ref{T} and subsequent counterexamples show that the crucial direction (i)\!\! $\implies$\!\! (ii) holds for a subset of $g$ with $g''(0)<0$, characterised by $g$ having a certain asymptotic property at $\pm\infty$, namely
\begin{equation}\label{RAE}
(\inf\dom g>-\infty) \vee \left(\liminf_{x\downarrow-\infty}\frac{x g'_+(x)}{g(x)}>1\right)\qquad\text{and}\qquad \limsup_{x\uparrow\infty}\frac{x g'_-(x)}{g(x)}<1, 
\end{equation}
which was previously introduced and employed in a different context by \citet{kramkov.schachermayer.99} and \citet{schachermayer.01} under the name of \emph{reasonable asymptotic elasticity (RAE)}. The paper provides counterexamples how  (i) \!\!$\iff$\!\! (ii) fails for $g$ bounded above, or how (ii)\!\! $\implies$\!\! (i) fails for $g$ unbounded above, when RAE~\eqref{RAE} is replaced by the weaker Inada condition~\eqref{eq:Inada}. Theorem~\ref{T} also provides corresponding results for predictable functions $g$.

The paper is organised as follows. Section~\ref{S:2} sets up the notation and introduces local expected utility. Section~\ref{S:3} establishes generic existence of local expected utility optimizers under the Inada condition 
\begin{equation}\label{eq:Inada}
\frac{1}{g'(-\infty)}=g'(\infty)=0,
\end{equation} 
where, as is common in convex analysis, we set $g'=\infty$ on $(-\infty,\inf\dom g)$.  
A generic counterexample shows that without \eqref{eq:Inada} existence cannot be guaranteed in general.   
Section~\ref{S:4} contains the main result of the paper. An illustrative application is discussed in Section~\ref{S:5}. Counterexamples that demonstrate tightness of the main result are collected in Section~\ref{S:6}. Appendix~\ref{Appx:A} contains measurable selection arguments; Appendix~\ref{Appx:B} provides results on second-order Peano differentiability; Appendix~\ref{Appx:C} develops calculus of ``variations'' on characteristic  triplets and studies optimization of convex functions generated by such ``variations''; and Appendix~\ref{Appx:D} provides the proof of Lemma~\ref{L:T}.

\section{Setup}\label{S:2}
Throughout, we fix a finite time horizon $T<\infty$ and work with all
processes stopped at $T$. 

\subsection{Local utility functions}
\begin{definition}
   We say that a predictable $g:\R\to[-\infty,\infty)$ is a \emph{local utility function} if it is  concave, non-decreasing, upper semi-continuous, and twice Peano-differentiable at $0$ with $g(0)=0$. 
\end{definition}
\begin{remark}
By Proposition~\ref{P:rock}, the second Peano derivative of a concave function equals the classical second derivative computed on a punctured neighbourhood.
\end{remark}

\begin{example} \label{ex:LRT}
In applications, one commonly encounters local utility functions obtained by monotonizing the family of linear risk tolerance (LRT) utilities (see, e.g., \citet{cass.stiglitz.70})
\begin{equation*}
\gLRT(x) = 
\begin{cases}
\left.
\begin{array}[c]{cc}%
\frac{((1+\rt x)\vee0)^{1-1/\rt}-1}{\rt-1},& \phantom{0}\ \,\rt<0\phantom{\neq}\\[1ex]
1-\e^{-x}, & \ \rt=0
\end{array}
\right\}  & \qquad x\in\R,\\[4ex]
\left.
\begin{array}[c]{cc}%
\frac{(1+\rt x)^{1-1/\rt}-1}{\rt-1}, & 0<\rt\neq 1\\[1ex]
 \makebox[0pt][l]{\quad\ \,$\log (1+x),$} \phantom{\frac{((1+\rt x)\vee0)^{1-1/\rt}-1}{\rt-1}} & \,\rt = 1
\end{array}
\right\} & \qquad x>-\frac{1}{\rt}.%
\footnotemark 
\end{cases}
\end{equation*}
\footnotetext{To ensure upper semicontinuity for $\rt>0$, one takes $\gLRT=-\infty$ on $(-\infty,-\frac{1}{\rt})$, $\gLRT(-\frac{1}{\rt})=-\infty$ for $\rt\in(0,1]$,  and $\gLRT(-\frac{1}{\rt})=\frac{-1}{\rt-1}$ for $\rt>1$.}
%
Here, the local utility function \eqref{eq:gMMV} from the introduction is obtained by taking $\rt=-1$.
\end{example}

\subsection{Variations}
The following setup is adapted from \citet*{crs.25.aheadofprint}. 
Fix $d\in\N$ and consider an $\R^d$-valued semimartingale $R$. We write $L(R)$ for the set of predictable processes that integrate $R$. For $\zeta\in L(R)$, we often use the shortcut notation $\zeta\sint R = \int_0^\cdot \zeta_t\d R_t$. For a time-constant deterministic $g:\R\to[-\infty,\infty]$ with $g(0)=0$, twice Peano-differentiable at $0$, and $\zeta\in L(R)$ such that $\zeta\Delta R\in\dom g$, the process 
$$g\circ (\zeta\sint R) = g'(0)
\zeta\sint R +\frac{1}{2}g''(0)[\zeta\sint R,\zeta\sint R]^c + \sum_{t\leq\cdot}\bigs(g(\zeta_t\Delta R_t)-g'(0)\zeta_t\Delta R_t\bigs) $$
is the $g$-variation of $\zeta\sint R$ by \citet[Th\'eor\`eme~2]{emery.78}. For predictable functions, the same operation can be defined under suitable conditions by the results in \citet*{cerny.ruf.22.ejp}; below it will only be needed for the process-level interpretation of the corresponding drift rates. We write $\id$ for the identity function, i.e., $\id(x)=x$ and $\id_1$, $\id_2$, etc.~for its components, i.e., $\id_1(x)=x_1$, $\id_2(x)=x_2$, and so on. We also fix a universal truncation function $\tr$ by setting 
\begin{equation}\label{tr}
\text{$\tr_1 = \id_1\indicator{\abs{\id_1}\leq 1}$, $\tr_2 = \id_2\indicator{\abs{\id_2}\leq 1}$, and so on,}
\end{equation}
and write  $\tr\circ R$ for the process $R$ with large jumps removed.

\subsection{Semimartingale characteristics and drift formulae}
By \citet*[II.2.6--9]{js.03}, there exists a ``good version'' of semimartingale characteristics and a predictable non-decreasing locally integrable process $A$ such that the characteristics of $R$ are absolutely continuous with respect to $A$. We denote the corresponding predictable differential characteristics by $(b^{\tr\circ R}, c^R, F^R)$. If the integrals on the right-hand side are defined and finite, we sometimes write 
\begin{align*}
 b^{R[0]}={}& b^{\tr\circ R}-\int_{\R^d}\tr(x)F^R(\d x),\\
 b^R={}& b^{\tr\circ R}+\int_{\R^d}(x-\tr(x))F^R(\d x).
\end{align*}

All (in)equalities involving processes
are assumed to hold $(\P\times A)$-almost everywhere. 

\begin{definition}\label{D:260612}
We say that a predictable function $\xi$ is compatible with $R$ if
$F^R((\dom\xi)^c)=0$. We say that a twice Peano-differentiable predictable function $\xi$  with
$\xi(0)=0$ is
$\sigma$-locally $\sigma$-special relative to $R$ if it is compatible
with $R$ and
\[
\int_{\R^d}\abs{\xi(x)}
\indicator{\abs{\xi(x)}>1}F^R(\d x)<\infty.
\]
In this case, we write
\[
b^{\xi,R}
=
D\xi(0)b^{\tr\circ R}
+\frac{1}{2}\trace(D^2\xi(0)c^R)
+\int_{\R^d}\bigs(\xi(x)-D\xi(0)\tr(x)\bigs)F^R(\d x).
\]
\end{definition}

The terminology ``$\xi$ is $\sigma$-locally $\sigma$-special relative to $R$'' has the following process-level interpretation. If
$(D_n)_{n\in\N}$ is a $\sigma$-localising sequence such that
$\xi\circ(\indicator{D_n}\sint R)$ is well defined, then these processes
are $\sigma$-special and
\[
b^{\xi\circ(\indicator{D_n}\sint R)}
=
\indicator{D_n}b^{\xi,R},
\qquad n\in\N;
\]
see \citet*[Lemma~D.1]{crs.25.aheadofprint}. Thus $b^{\tr\circ R}=b^{\tr,R}$, and $b^R=b^{\id\circ R}=b^{\id,R}$ if $R$ is special, or more generally, $\sigma$-special. 

\subsection{Instantaneous absence of arbitrage}
The following definition formulates a mild no-arbitrage assumption weaker than the existence of an equivalent $\sigma$-martingale measure for $R$. 

\begin{definition}\label{D:INA}
We say that a given triplet $(b^{\tr\circ R}, c^R, F^R)$ 
is \emph{instantaneously arbitrage-free (INA) at $(\omega,t)$} if for all   
$\lambda\in\R^d$ one has 
\begin{equation} \label{eq:INA}
\left.
\begin{gathered}
 \left(\lambda c^R_t = 0 \text{ and } \int_{\R^d} \indicator{\lambda x< 0}F^R_t(\d x) = 0\right) \quad \implies \\
\left(\int_{\R^d} \indicator{0<\lambda x}F^R_t(\d x) > 0 \text{ and }
\lambda b^{\tr\circ R}_t-\int_{\R^d}  \lambda \tr(x) F^R_t(\d x)<0\right)\\
\quad\text{or}\quad \\
\ \left(\int_{\R^d} \indicator{0<\lambda x}F^R_t(\d x) = 0 \text{ and } \lambda b^{\tr\circ R}_t - \int_{\R^d}  \lambda \tr(x) F^R_t(\d x)=0\right).
\end{gathered}
\qquad\right\}
\end{equation}
We say that $R$ is instantaneously arbitrage-free if its triplet is INA  $(\P\times A)$-everywhere.
\end{definition}
Intuitively,  Definition~\ref{D:INA} requires that the process $\lambda R$ be free to move in both directions (up and down) or else be constant at any moment in time for any $\lambda\in\R^d$. At predictable times $\tau$ when $R$ jumps with non-zero probability, the above condition is tantamount to the classical one-period no-arbitrage requirement
\begin{equation}\label{eq:INA_tau}
\P[\lambda \Delta R_\tau \geq 0|\F_{\tau-}]=1 \,\,\Longrightarrow \,\, \P[\lambda \Delta R_\tau=0|\F_{\tau-}]=1. 
\end{equation}

\subsection{Local expected utility}Fixing a local utility function $g$, by \emph{local expected utility} we understand the  function $\gfrak:\R^d\to [-\infty,\infty]$ given by
\begin{equation}\label{eq:locutil}
 \gfrak(\lambda) =g'(0) b^{ \tr\circ R}\lambda + \frac{1}{2}g''(0)\lambda c^{ R}\lambda^\top
+\int_{\R^d} \bigs(g(\lambda x)-g'(0)\lambda \tr(x)\bigs)F^{ R}(\d x)
\end{equation}
with the convention that $\gfrak(\lambda)=-\infty$ whenever $\int_{\R^d} g(\lambda x)\indicator{g(\lambda x)<-1}F^{ R}(\d x)=-\infty$, 
even if at the same time $\int g(\lambda x)\indicator{g(\lambda x)>1}F^{ R}(\d x)=\infty$; cf.\ \citet[Section~14.F]{rockafellar.wets.98} for an analogous integral convention in the context of convex rather than concave functions. Whenever $\gfrak(\lambda)\in\R$, the function $g(\lambda\id)$ is
$\sigma$-locally $\sigma$-special relative to $R$ and
$b^{g(\lambda\id),R}=\gfrak(\lambda)$.

Concavity of $g$ yields that $\gfrak$ is concave but not necessarily proper because it may take the value of $+\infty$ if $g$ is unbounded.
We let
\begin{align*}
  \edom\gfrak_t ={}& \{\lambda\in\R^d:\gfrak_t(\lambda)>-\infty\},\\
\dom\gfrak_t = {}& \{\lambda\in\R^d:\gfrak_t(\lambda)\in\R\}.  
\end{align*}
The set $\edom\gfrak_t$ denotes the \emph{effective domain} of the concave function $\gfrak_t$; see \citet[Section~4]{rockafellar.70}.
By \citet[Proposition~1.15]{kallenberg.17}, $\gfrak$ is a predictable function in the sense of \citet[II.1.4]{js.03}. The set of contingencies where $\gfrak$ does not take value $+\infty$ for any $\lambda\in\R^d$ is given by 
\begin{equation*}
 \D=\{(\omega,t):\gfrak_t(\lambda)<\infty \text{ for all }\lambda\in \R^d\}.   
\end{equation*} 
On $\Omega\times[0,T]$, one has $\D=\Omega\times[0,T]$ if $g$ is bounded above. By Lemma~\ref{L:260826London}, 
$\D$ admits a predictable version, which we fix henceforth.

\section{Existence and characterisation of predictable local optimizers}\label{S:3}

Hereafter, $g$ denotes a local utility function, which, being concave and suitably differentiable, entails $g''(0)\leq0$. Observe, however, that with $g''(0)=0$, one cannot hope \eqref{eq:locutil} to have a maximizer over $\lambda$ in general because with $R$ continuous, the local expected utility is linear, i.e., $\gfrak(\lambda) = g'(0)\lambda b^R$, $\lambda\in\R^d$.

\begin{proposition}\label{P}
The following statements hold.
\begin{enumerate}[(1)]
\item\label{P.1} If $\indicator{\D}\sint R$ is instantaneously arbitrage-free, $g''(0)<0$ on $\D$, and the Inada condition~\eqref{eq:Inada} holds on $\D$, then there exists a predictable process that maximizes $\gfrak$. 
\item\label{P.2}  If $\gfrak<\infty$, then  a predictable $\lopt$ maximizes $\gfrak$ if and only if  $\gfrak(\lopt)\in\R$ and
\[\xi^\lambda = g'_+(\lopt\id)((\lambda-\lopt)\id)^+-g'_-(\lopt\id)((\lambda-\lopt)\id)^-\]
is $\sigma$-locally $\sigma$-special relative to $R$ for all $\lambda\in\edom\gfrak$, and  
\begin{align}  \label{eq:260507}
-\infty<b^{\xi^\lambda, R}\leq 0,\quad \lambda\in\edom\gfrak.
\end{align}
 In such case (with  $\lambda=0$), the predictable function
$(g'_-\id^+-g'_+\id^- )(\lopt\id)$ is $\sigma$-locally $\sigma$-special relative to $R$ with 
\begin{align} \label{eq:260508}
0 \leq b^{(g'_-\id^+-g'_+\id^- )(\lopt\id),R}< \infty.
\end{align} 
\item\label{P.3} Suppose $g$ is strictly concave and satisfies $g''(0) < 0$ and $\gfrak<\infty$. If 
$\lopt,\widetilde\lambda$ are two predictable local optimizers, then $\lopt-\widetilde\lambda\in L(R)$ and 
$
    (\lopt-\widetilde\lambda)\sint R=0
$;
that is, any two predictable local optimizers differ by a null
strategy.
\end{enumerate}
\end{proposition}
\begin{proof}\ref{P.1}\
Lemma~\ref{L:260830London}\ref{London.2} in conjunction with Lemma~\ref{L:260826London} yield the statement. 

\ref{P.2}\ Lemma~\ref{L:FOC Levy} applied pointwise yields the statement.

\ref{P.3}\ Assuming $g$ is strictly concave,  Lemma~\ref{L:260830London}\ref{London.1} yields that $\lopt-\widetilde\lambda$ takes values in 
\begin{align*}
\mathcal{N} = \left\{
\lambda\in\R^d:
\lambda b^{\tr\circ R}
-
\int_{\R^d}\lambda \tr(x)F^R(\d x)=0,\quad
\lambda c^R=0,\quad
\int_{\R^d}\indicator{\lambda x\ne0}F^R(\d x)=0
\right\}.
\end{align*}
We claim that any predictable $\mathcal{N}$-valued $\zeta$ belongs to $L(R)$. Indeed, for such $\zeta$ one has 
\begin{align*}
  b^{(\zeta\id)^2\wedge1,R} ={}& \zeta c^R \zeta^\top + \int_{\R^d}((\zeta x)^2\wedge 1)F^R(\d x) = 0,\\
b^{\tr(\zeta\id),R} ={}& \zeta b^{\tr\circ R} +   \int_{\R^d}(\tr(\zeta x)-\zeta\tr(x))F^R(\d x) = 0,
\end{align*}
hence $\zeta\in L(R)$ by Proposition~\ref{P:int}. Since $\zeta\sint R$ has no jumps, no continuous quadratic variation, and zero drift, it is a null process. 
\end{proof}
The next proposition and example illustrate the importance of the Inada condition~\eqref{eq:Inada} in  Proposition~\ref{P}. Observe that for time-constant deterministic $g$, the requirement $g'(0)>0$ is weaker than 
$g''(0)<0$.

\begin{proposition}
    Suppose 
    $g$ is time-constant deterministic 
    with $g'(0)>0$. If $g$ does not satisfy the Inada condition~\eqref{eq:Inada}, then there exists an arbitrage-free special $\R$-valued L\'evy process $R$ with finite L\'evy measure such that $\gfrak<\infty$ and $\gfrak$ does not attain its supremum.
\end{proposition}
\begin{proof}
Consider a time-constant deterministic $g$ with $g'(0)>0$ such that the Inada condition~\eqref{eq:Inada} does not hold. Without loss of generality, we may assume $g'(0)=1$. We now construct a one-dimensional arbitrage-free L\'evy process $R$ with finite L\'evy measure 
 to witness the assertion. To this end, take $d=1$ (a single risky asset),  $c^R=0$,
$$ F^R(\d x) = \left(\indicator{0<x\leq 1} + \frac{1}{(g'(-\infty))^2}\indicator{-1\leq x <0}\right)\d x, $$
and 
$$b^{R[0]} =  \frac{1}{g'(-\infty)} - \frac{g'(\infty)}{4}$$
relative to the activity process $A_t=t$. This yields

\[
b^R = b^{R[0]}+\int_\R xF^R(\d x)
= \frac{1}{g'(-\infty)}\left(1-\frac{1}{2}\frac{1}{g'(-\infty)}\right) + \frac{1}{4}\left(2-g'(\infty)\right)>0
\]
is finite, hence 
\begin{equation}\label{eq:260502Granada}
\gfrak(\lambda) = b^R\lambda + \int_\R(g(\lambda x)-\lambda x)F^R(\d x)<\infty,\qquad \lambda\in\R,  
\end{equation}
in view of $g-\id\leq0$.

For $g'(-\infty)<\infty$, $F^R$ charges both positive and negative values, hence INA holds. 
Furthermore,  $R$ admits an equivalent martingale measure. Indeed, there exist $a,b>0$ such that
\[
b^{R[0]}
+
a\int_{\R}x\mathbf 1_{x\in(0,1]}F^R(\d x)
+
b\int_{\R}x\mathbf 1_{x\in[-1,0)}F^R(\d x)
=0,
\]
since the first integral is strictly positive and either the second is
strictly negative or $b^{R[0]}<0$. Hence
\[
F^{R,\Q}(\d x)
=
\left(
a\indicator{x>0}
+
b\indicator{x<0}
\right)F^R(\d x)
\]
defines an equivalent change of measure under which $R$ is a martingale.

Observe that $b^R>0$ and \eqref{eq:260502Granada} give $\gfrak(\lambda)\leq 0 $ for $\lambda\leq 0$. For $\lambda=1$, Lemma~\ref{L:260422Granada} yields 
\begin{align*}
\lim_{y\uparrow\infty}\frac{\gfrak(y)}{y} ={}& 
b^{R[0]}+g'(\infty)\int_\R x\indicator{x>0}F^R(\d x) + g'(-\infty)\int_\R x\indicator{x<0}F^R(\d x)\\ ={}& \frac{g'(\infty)}{4} + \frac{1}{2g'(-\infty)} >0,  
\end{align*} 
hence $\gfrak$ is unbounded on $[0,\infty)$ and 
its supremum is therefore not attained.
\end{proof}

\begin{example}
With \citet{kallsen.99} let 
\begin{equation*}
 g = 1+\id - \sqrt{1+\id^2}.
\end{equation*}
One easily verifies that $g$ is a local utility function and that, furthermore, it is strictly increasing, strictly concave with $g''(0)=-1<0$, and bounded above. However, $g$ does not satisfy the Inada condition~\eqref{eq:Inada} as $g'(-\infty)=2<\infty$. 

Let $R$ be an $\R$-valued L\'evy process with $b^{R[0]}=1$, $c^R=0$, and $F^R(\d x) = \indicator{|x|\leq 1}\d x$ relative to the activity process $A_t=t$. We then have 
$$\gfrak(\lambda) = \lambda + \int_{-1}^1\left(1+\lambda x-\sqrt{1+\lambda^2x^2}\right)\d x,\qquad \lambda\in\R, $$
which after simplifications yields $\gfrak(0)=0$ and
$$\gfrak(\lambda) = 2+\lambda - 2\lambda^{-1}\int_0^\lambda\sqrt{1+y^2}\mkern2mu\d y 
= 2 - \lambda^{-1} \log\left(\lambda+\sqrt{1+\lambda^2}\right)+\lambda-\sqrt{1+\lambda^2}, \qquad \lambda \neq 0.$$
Therefore, $\sup_{\lambda\in\R}\gfrak(\lambda)=2$ is not attained. For $b^{R[0]}>1$, the supremum even fails to be finite.
\end{example}

\section{Integrability of the local optimizer}\label{S:4}

We remind the reader that $g$ denotes a local utility function, $R$ a semimartingale, and $\gfrak$ the resulting local expected utility. We wish to determine circumstances under which the local optimizer $\lopt$, whose existence we established in Proposition~\ref{P}, integrates $R$. Denote by $(\gfrak^\rt)_{\rt\in\R}$ the family of predictable local expected utility mappings generated by the local utility functions in Example~\ref{ex:LRT} via \eqref{eq:locutil}. 
The following is known to be true.  
\begin{theorem}[{\cite[Theorem~2.1]{crs.25.aheadofprint}}]\label{T:CRS}
Suppose $R$ is instantaneously arbitrage-free and has independent increments.
For $\rt=-1$, i.e., for the local utility function 
\begin{equation}\label{eq:gMMV2}
  \id\wedge1 - \frac{1}{2}(\id\wedge1)^2,  
\end{equation}
$\gfrak^\rt$ admits a predictable local optimizer, say $\lLRT$, and the following are equivalent. 
\begin{enumerate}[(i)]
\item\label{T2.i} The cumulative maximal local expected utility is finite, i.e., 
$\gfrak^{\rt}(\lLRT) \sint A_T<\infty$.
\item\label{T2.ii} The strategy $\lLRT$ integrates $R$.
\end{enumerate}
\end{theorem}
The aim of the paper is to understand whether such tight equivalence holds in a broader class of utility functions and, in particular, for other values of $\rt\in\R$ with $\gLRT$ from Example~\ref{ex:LRT}. 

The following integrability criterion, previously used to establish Theorem~\ref{T:CRS}, will be instrumental below. It  is a slight modification of \citet*[Theorem~III.6.30]{js.03} reformulated in terms of the drift rates introduced in Definition~\ref{D:260612}, featuring variation functions $\id^2\wedge1$ and $\tr=\id\indicator{\abs{\id}\leq1}$ similar to those appearing in  the local utility function \eqref{eq:gMMV2}.
\begin{proposition}[{\cite[Corollary~H.1]{crs.25.aheadofprint}}]
\label{P:int}
Let $\zeta$ be an $\R^d$-valued predictable process. 
Then
\[
\zeta\in L(R)
\quad\Longleftrightarrow\quad
\left(
b^{(\zeta\id)^2\wedge1,R}
+
\abs{b^{\tr(\zeta\id),R}}
\right)\sint A_T<\infty.
\]
\end{proposition}

For $g$ bounded above, this immediately yields the reverse direction of the equivalence we seek. Here the boundedness from above is essential, as shown in Example~\ref{E2}. 
\begin{corollary}\label{C}
Suppose $g'(0)$ and $g(\infty)$ are locally bounded and let 
$\zeta\in L(R)$. Then
$$ (\gfrak(\zeta)\vee0)\sint A_T<\infty.$$
\end{corollary}
\begin{proof}
We have
$
g
\leq
g'(0)\tr
+
g(\infty)(\id^2\wedge1)$.
Hence
\[
\gfrak(\zeta)
= \lim_{H\uparrow\infty}
b^{(-H)\vee g(\zeta\id),R}
\leq
g'(0)b^{\tr(\zeta\id),R}
+
g(\infty)b^{(\id^2\wedge1)(\zeta\id),R},
\]
and therefore
\[
\gfrak(\zeta)\vee0
\leq
g'(0)\bigl|b^{\tr(\zeta\id),R}\bigr|
+
g(\infty)b^{(\id^2\wedge1)(\zeta\id),R}.
\]
The assertion now follows from Proposition~\ref{P:int}.
\end{proof}

 To prepare for the statement of the main theorem, we now introduce a lemma that describes local behaviour of $g$ and related quantities near zero. Its proof is found in Appendix~\ref{Appx:D}. These considerations become important when $g$ is not time constant and deterministic.   

 \begin{lemma}\label{L:T}
Let 
\begin{align} \label{eq:260924.1}
    k=g'_-\id^+-g'_+\id^-,
\end{align}
with the convention that
$g-k$ is equal to $\infty$ outside of $\dom g$. If $g''(0)<0$,
then there exist predictable processes $\alpha>0$,
$\beta>0$, and $\gamma\geq1$ such that
\begin{align}
&\ -2\alpha\in\dom g,\label{eq:alpha in dom}\\
\beta\id^2\indicator{\abs{\id}\leq\alpha}
 \leq{}(g-k)&\indicator{\abs{\id}\leq\alpha},
 \qquad
 \bigl(g'(0)\id-g\bigr)\indicator{\abs{\id}\leq\alpha}
 \leq \gamma(g-k)\indicator{\abs{\id}\leq\alpha}.
 \label{eq:260519}
\end{align}
 \end{lemma}
The proof of the next theorem can be found  at the end of this section.
\begin{theorem}\label{T}
Suppose $g''(0)<0$, $R$ is instantaneously arbitrage-free, and 
$\gfrak<\infty$. Suppose there exist predictable
processes $\alpha>0$, $\beta>0$, $\gamma\geq1$, and $\delta\geq1$ such that
\eqref{eq:alpha in dom} and \eqref{eq:260519} hold,
\begin{equation}\label{eq:tail-control}
    \abs{g} \indicator{\abs{\id}>\alpha}\leq \delta(g-k)\indicator{\abs{\id}>\alpha},
\end{equation}
where $k$ is defined in \eqref{eq:260924.1}, and
\begin{equation}\label{eq:uniformity}
    \frac{1}{\alpha},
    \qquad
    \frac{1}{\beta},
    \qquad
    \frac{\gamma+\delta}{g'(0)}\qquad \text{are locally bounded.}
\end{equation}
 Then a predictable local optimizer, denoted by
$\lopt$, exists, and
\begin{enumerate}[(i)]
\item\label{T.i}
$\gfrak(\lopt)\sint A_T<\infty$
\end{enumerate}
implies
\begin{enumerate}[resume*]
\item\label{T.ii}
$\lopt$ integrates $R$.
\end{enumerate}

If $g$ is time-constant and deterministic, the assumptions
\eqref{eq:alpha in dom}--\eqref{eq:uniformity} can be replaced by the reasonable asymptotic elasticity condition \eqref{RAE}, i.e.,
\begin{equation}\label{RAE-det}
\left(
    \inf\dom g>-\infty
\right)
\quad\text{or}\quad
\left(
    \liminf_{x\downarrow-\infty}
    \frac{k(x)}{g(x)}>1
\right),
\qquad
\limsup_{x\uparrow\infty}
\frac{k(x)}{g(x)}<1,
\end{equation}
in which case $\delta$ can be chosen time-constant and deterministic.
\end{theorem}
The next corollary introduces a functional form of $g$ for which the assumptions of Theorem~\ref{T} are easy to verify.
\begin{corollary}
Suppose $g$ is of the form 
    \[g = \eta\widetilde{g}(\zeta\id),\]
    with predictable $\eta>0,\zeta>0$ and $\widetilde{g}$ a time-constant deterministic local utility function that satisfies the reasonable asymptotic elasticity condition \eqref{RAE} and $\widetilde{g}''(0)<0$. Then the assumptions
\eqref{eq:alpha in dom}--\eqref{eq:uniformity} can be replaced by
\[\zeta
\qquad\text{and}\qquad
\frac{1}{\eta\zeta^2}\qquad\text{are locally bounded}\]
or, equivalently,
\[\frac{g''(0)}{g'(0)}\qquad\text{and}\qquad \frac{1}{g''(0)}\qquad\text{are locally bounded,}\]
and $\delta$ can be chosen time-constant and deterministic.
\end{corollary}
\begin{proof}
    The second part of Theorem~\ref{T} applied to $\widetilde{g}$ yields constants $\widetilde\alpha>0$, $\widetilde\beta>0$, $\widetilde\gamma\geq1$, and $\widetilde\delta\geq1$ such that \eqref{eq:alpha in dom}--\eqref{eq:tail-control} hold with \ $\widetilde{}$ \ decorations. It follows by definition of $g$ that \eqref{eq:alpha in dom}--\eqref{eq:tail-control} hold, this time without decorations, for 
    \[\alpha = \frac{\widetilde\alpha}{\zeta} ,\quad \beta=\eta\zeta^2\widetilde\beta,\quad \gamma = \widetilde\gamma,\quad \text{and}\quad \delta = \widetilde\delta. \]
    Moreover,
    \[
    \frac{1}{g'(0)}
    =
    \frac{\zeta}{\widetilde g'(0)}
    \frac{1}{\eta\zeta^2},
    \]
    so \eqref{eq:uniformity} follows from the assumed local boundedness. An application of the first part of Theorem~\ref{T} concludes.
\end{proof}

Examples~\ref{E3} and \ref{E4} show that, for deterministic $g$ with
$\inf\dom g=-\infty$, the strict asymptotic elasticity conditions in
Theorem~\ref{T} cannot in general be replaced merely by the Inada
condition~\eqref{eq:Inada}.

Combining Corollary~\ref{C} and Theorem~\ref{T}, we obtain the
following answer to the initial query.

\begin{corollary}
Theorem~\ref{T:CRS} holds without the assumption of independent increments and with $\rt=-1$ replaced by $\rt\in(-\infty,1)$ in the notation of Example~\ref{ex:LRT}. Furthermore, the implication from \ref{T2.i} to \ref{T2.ii} in Theorem~\ref{T:CRS} holds for all $\rt\in\R$. 
\end{corollary}
\begin{proof}
One only needs to check reasonable asymptotic elasticity. To this end, letting $\xinfLRT = \inf \dom \gLRT$,  one has 
\begin{equation*}
1<\lim_{x\downarrow\xinfLRT}\frac{x\gLRT'(x)}{\gLRT(x)} = 
\begin{cases}
1-\frac{1}{\rt}, & \qquad\rt<0,
\\[2ex]
\ \,\infty,      & \qquad\rt\geq0,
\end{cases}
\end{equation*}
and 
\[
1>
\lim_{x\uparrow\infty}
\frac{x\gLRT'(x)}{\gLRT(x)}
=
1-\frac{1}{\rt\vee1},\qquad \rt\in\R.
\]
An application of Corollary~\ref{C} and Theorem~\ref{T} concludes.
\end{proof}
\begin{proof}[Proof of Theorem~\ref{T}]
The monotonicity of $g'_+$ and $g'_-$ yields that
\[
g-k
\text{ is non-increasing on }(-\infty,0]
\text{ and non-decreasing on }[0,\infty).
\]
Hence, the first inequality in \eqref{eq:260519} gives
\begin{equation}\label{eq:260519a}
    \beta(\id^2\wedge\alpha^2)\leq g-k.
\end{equation}
Moreover, \eqref{eq:260519} yields
\begin{equation}\label{eq:260519b}
g\indicator{\abs{\id}\leq\alpha}
\leq
g'(0)\id\indicator{\abs{\id}\leq\alpha}
\leq
g\indicator{\abs{\id}\leq\alpha}
+
\gamma(g-k).
\end{equation}
By \eqref{eq:tail-control} and the piecewise monotonicity of $g-k$,
\begin{equation}\label{eq:260519c}
    \abs{g}\indicator{\abs{\id}>\alpha}
    \leq \delta(g-k).
\end{equation}
This implies RAE pointwise for (almost) all $(\omega,t)$, which in turn yields the Inada condition~\eqref{eq:Inada}. Thus Proposition~\ref{P} yields a predictable local optimizer, here denoted by $\lopt$.

By the local optimality of $\lopt$ and $\gfrak<\infty$, we have
$\gfrak(\lopt)\in\R$. Hence, by \eqref{eq:locutil} and
\eqref{eq:260508},
\[
    0\leq b^{g(\lopt\id),R}
    = \gfrak(\lopt)<\infty,
    \qquad
    0\leq b^{k(\lopt\id),R}<\infty.
\]
Hence, by Lemma~\ref{L:B2}\ref{L:B2.3},
\begin{equation}\label{eq:260519d}
    0\leq
    b^{(g-k)(\lopt\id),R}
    = b^{g(\lopt\id),R}-b^{k(\lopt\id),R}
    \leq
    b^{g(\lopt\id),R}.
\end{equation}

Now \eqref{eq:260519a}, \eqref{eq:260519d}, and
Lemma~\ref{L:B2}\ref{L:B2.3} yield
\begin{equation}\label{eq:midQV}
    \beta
    b^{(\id^2\wedge\alpha^2)(\lopt\id),R}
    \leq
    b^{(g-k)(\lopt\id),R}
    \leq b^{g(\lopt\id),R}.
\end{equation}
Furthermore, \eqref{eq:260519b}, the triangle inequality,
\eqref{eq:260519c}, and Lemma~\ref{L:B2}\ref{L:B2.3} give
\begin{equation*}
\begin{split}
g'(0)
\left|
b^{(\id\indicator{\abs{\id}\leq\alpha})(\lopt\id),R}
\right|
&\leq
\left|
b^{(g\indicator{\abs{\id}\leq\alpha})(\lopt\id),R}
\right|
+
\gamma b^{(g-k)(\lopt\id),R}
\\
&\leq
b^{g(\lopt\id),R}
+
\left|
b^{(g\indicator{\abs{\id}>\alpha})(\lopt\id),R}
\right|
+
\gamma b^{(g-k)(\lopt\id),R}
\\
&\leq
(1+\gamma+\delta)b^{g(\lopt\id),R}.
\end{split}
\end{equation*}
Since $\gamma\geq1$, we thus have
\begin{equation}\label{eq:260519f}
b^{(\id^2\wedge\alpha^2)(\lopt\id),R}
+
\left|
b^{(\id\indicator{\abs{\id}\leq\alpha})(\lopt\id),R}
\right|
\leq
\left(
    \frac{1}{\beta}
    +
    \frac{2(\gamma+\delta)}{g'(0)}
\right)
b^{g(\lopt\id),R}.
\end{equation}

Since $1/\alpha$ is locally bounded, the truncation at $\alpha$ in
\eqref{eq:260519f} can be replaced by the truncation at $1$ in
Proposition~\ref{P:int}.
In view of
\begin{equation}\label{eq:gfrak lopt}
    b^{g(\lopt\id),R}
    =
    \gfrak(\lopt),
\end{equation}
\eqref{eq:uniformity}, \eqref{eq:260519f}, and \ref{T.i} yield
\[
\left(
b^{(\id^2\wedge1)(\lopt\id),R}
+
\left|b^{\tr(\lopt\id),R}\right|
\right)\sint A_T<\infty.
\]
Proposition~\ref{P:int} therefore yields \ref{T.ii}.

It remains to consider time-constant deterministic $g$. By twice
Peano differentiability at zero, constants $\alpha>0$, $\beta>0$,
and $\gamma\geq1$ can be chosen such that
\eqref{eq:alpha in dom} and \eqref{eq:260519} hold. 

Suppose first that $\xinf=\inf\dom g>-\infty$. If $g(\xinf)$ is finite,
then $\abs{g}$ is bounded on $[\xinf,-\alpha]$, whereas
\eqref{eq:260519} and the monotonicity of $g-k$ yield
$g-k\geq\beta\alpha^2$ on $[\xinf,-\alpha]$.
Hence $k/g$ is bounded away from $1$ there, which yields a finite $\delta_->0$ such that
\begin{equation}\label{260816Tenby}
    -g(x)\leq \delta_-\bigl(g(x)-k(x)\bigr),
    \qquad x\in\dom g,\quad x<-\alpha.
\end{equation}
If instead $g(\xinf)=-\infty$, fix $y\in(\xinf,0)$. By concavity, for $x \in (\xinf, y)$,
\[
    g'_+(x)\geq\frac{g(y)-g(x)}{y-x},
\]
and therefore
\[
    \liminf_{x\downarrow\xinf}\frac{k(x)}{g(x)}
    \geq
    \frac{-\xinf}{y-\xinf}>1.
\]
Thus $k/g$ is bounded away from $1$ near $\xinf$, while on the
remaining compact subinterval of $(\xinf,-\alpha]$ it is bounded away from $1$ by the preceding argument. Hence again there is
$\delta_->0$ such that \eqref{260816Tenby} holds. 
Suppose finally that $\xinf=-\infty$. Then 
\[
\liminf_{x\downarrow\xinf}\frac{k(x)}{g(x)}>1
\]
follows from the left-hand condition in \eqref{RAE-det} and we again obtain $\delta_->0$ such that \eqref{260816Tenby} holds. 

For $x>\alpha$, the right-hand condition in \eqref{RAE-det} yields a
constant $\delta_+>0$ such that
\[
    g(x)\leq \delta_+\bigl(g(x)-k(x)\bigr),
    \qquad x>\alpha.
\]

Taking $ \delta=1\vee \delta_+\vee \delta_-$
gives \eqref{eq:tail-control}. Since $\alpha$, $\beta$, $\gamma$, $\delta$, and $g'(0)$ are now
time-constant and deterministic, \eqref{eq:uniformity} follows automatically.
\end{proof}

The next remark provides a 
process-level interpretation of finite cumulative local expected utility in \ref{T.i}. In optimal portfolio applications, it is helpful to know whether the processes $g\circ (\lopt\sint R)$ and $k\circ (\lopt\sint R)$ are special. 

\begin{remark}
Suppose the assumptions of Theorem~\ref{T} are satisfied and \ref{T.i} therein holds. Then $g$ and $k$  are  
$\sigma$-locally $\sigma$-special relative to $\lopt\sint R$ by  \eqref{eq:260519c}, \eqref{eq:260519d}, and
Definition~\ref{D:260612}.
Furthermore, if
$(g-k)\circ(\lopt\sint R)$ is well defined, then it is special by \ref{T.i}, \eqref{eq:midQV}, and the characterisation of special representations in \cite[Lemma~D.1]{crs.25.aheadofprint}.  If, moreover, $g\circ(\lopt\sint R)$ is well defined  and  $g'(0)$ is locally bounded, then $g\circ(\lopt\sint R)$ is special and by \eqref{eq:gfrak lopt} its predictable compensator $B^{g\circ(\lopt\sint R)}$ satisfies
\[
    B^{g\circ(\lopt\sint R)}
    =
    \gfrak(\lopt)\sint A.
\]

Indeed, on
$\{\abs{\id}\leq\alpha,\ \abs{g}>1\}$, 
\eqref{eq:260519b} yields
\[
\abs{g}\leq g'(0)\abs{\id}+\gamma(g-k).
\]
If $g'(0)\abs{\id}\geq1/2$, then
\[
g'(0)\abs{\id}\leq2g'(0)^2\id^2
\leq\frac{2g'(0)^2}{\beta}(g-k),
\]
whereas, if $g'(0)\abs{\id}<1/2$, then $\abs{g}>1$ implies
$\gamma(g-k)>1/2>g'(0)\abs{\id}$. Together with
\eqref{eq:260519c}, we therefore obtain 
\[
\abs{g}\indicator{\abs{g}>1}
\leq
\left(\frac{2g'(0)^2}{\beta}+2\gamma+\delta\right)(g-k).
\]
The factor in parentheses is locally bounded by the assumptions of the
Theorem~\ref{T} and the local boundedness of $g'(0)$. 
By \ref{T.i} therein and the characterisation of special representations in \cite[Lemma~D.1]{crs.25.aheadofprint}, $g\circ(\lopt\sint R)$ is special.
\end{remark}

\section{An application}\label{S:5}

We now discuss an application concerning globally optimal strategies for power utility, where the results of this study offer some simplifications. 
\begin{theorem}[{\citet{kallsen.muhle-karbe.10.aap}, Theorem~3.1}]\label{T:KMK10AAP}
Let $\mathbb F$ denote the underlying filtration.
Suppose $R$ admits differential characteristics with respect to $A_t=t$, there is an equivalent local martingale measure for $R$, and there exists a $\sigma$-field $\mathcal H$ such that $R$ is an
$\mathcal H$-PII and the differential characteristics of $R$ are
$\mathcal H$-measurable and coincide with the differential
characteristics with respect to the enlarged filtration
$\mathbb F\vee\mathcal H$.

Fixing $0<p\neq1$, assume there exists an $\R^d$-valued process $\piopt \in L(R)$ such that the following conditions are satisfied up to a $\P \otimes A$-null set on $\Omega \times [0,T]$.
\begin{enumerate}[(a)]
\item\label{T:KMK10AAP.a} $F^R(\{x \in \R^d: 1+\piopt x \leq 0\})=0$. 
\item\label{T:KMK10AAP.b} $\int \abs{x(1+\piopt x)^{-p}-\tr(x)}F^R(\d x)<\infty$. 
\item\label{T:KMK10AAP.c} For all $\pi \in \R^d$ such that 
\begin{equation}\label{solvency}F^R(\{x \in \R^d: 1+\pi x < 0\})=0,
\end{equation}
one has 
$$(\pi-\piopt)\left(b^{\tr\circ R}- pc^R \piopt^\top +\int \left(x(1+\piopt x)^{-p}-\tr(x)\right)F^R(\d x)\right) \leq 0.$$
\item\label{T:KMK10AAP.d} $\abs{\gfrak^{1/p}(p\piopt)}\sint A_T<\infty$.
\end{enumerate}
Then the strategy $\hat{\pi}\Exp(\hat{\pi} \sint R)_{-}$
is optimal, i.e.,  
$$\frac{\E[(1+\varphi\sint R_T)^{1-p}]}{1-p}\leq \frac{\E[(1+\hat{\pi}\Exp(\hat{\pi} \sint R)_{-}\sint R_T)^{1-p}]}{1-p}=\frac{\E\left[\exp\left((1/p-1)\gfrak^{1/p}(p\piopt)\sint A_T\right)\right]}{1-p}$$
for all $\varphi\in L(R)$ such that $1+\varphi\sint R>0$.  
\end{theorem}
Our aim is to highlight, based on the results in Sections~\ref{S:3} and \ref{S:4}, that the assumptions of Theorem~\ref{T:KMK10AAP} are implied by  an apparently weaker set of assumptions. To this end, fix $0<p\neq1$ and consider the local utility function $g=g_{1/p}$ from Example~\ref{ex:LRT} and the corresponding local expected utility $\gfrak$. 

Consider a predictable $\piopt$ that satisfies conditions
\ref{T:KMK10AAP.a}--\ref{T:KMK10AAP.d} but not necessarily
$\piopt\in L(R)$. Setting $\lopt = p\piopt$,
condition~\ref{T:KMK10AAP.d} yields $\gfrak(\lopt)\in\R$, while
conditions \ref{T:KMK10AAP.b} and \ref{T:KMK10AAP.c} yield
\eqref{eq:260507} for all $\lambda\in\edom\gfrak$.
Lemma~\ref{L:FOC Levy} in turn yields that $\lopt = p\piopt$ is a local
optimizer and $\gfrak<\infty$. Together with Theorem~\ref{T},
condition~\ref{T:KMK10AAP.d} then also yields $\lopt\in L(R)$, hence
$\piopt\in L(R)$.

We now turn to the question of existence of a strategy $\piopt$ that satisfies the assumptions of Theorem~\ref{T:KMK10AAP}. The existence of an equivalent local martingale measure is stronger than INA. Since $g$ satisfies the Inada condition~\eqref{eq:Inada} and $g''(0)<0$, Proposition~\ref{P} yields the existence of a predictable local optimizer $\lopt\in\edom\gfrak$. Letting $\piopt=\lopt/p$ and assuming further that \ref{T:KMK10AAP.d} holds, we have $\gfrak<\infty$. 
Since $\dom(\id g')=(-p,\infty)$ and $\lopt=p\piopt$,
\[
F^R\bigl(\{x:1+\piopt x\leq0\}\bigr) = 
F^R\bigl(\{x:\lopt x\notin\dom(\id g')\}\bigr)
=0
\]
by Proposition~\ref{P}\ref{P.2} with $\lambda=0$.
 Thus condition~\ref{T:KMK10AAP.a} holds.
 Moreover,  Proposition~\ref{P} and condition~\ref{T:KMK10AAP.b} yield 
\begin{equation}\label{FOC KMK}
-\infty<b^{g'(\lopt\id)(\lambda-\lopt)\id, R} = b^{g'(\lopt\id)\id, R}(\lambda-\lopt)\leq 0,\quad \lambda\in\operatorname{cl}(\edom\gfrak).
\end{equation}
 If $\pi$ satisfies the solvency condition~\eqref{solvency}, 
then $yp\pi\in\edom\gfrak$ for all $y\in[0,1)$ and thus $p\pi$ lies in the closure of $\edom\gfrak$. This and \eqref{FOC KMK} with $\lambda = p\pi$ shows \ref{T:KMK10AAP.c}.
Since $g$ also satisfies the RAE condition~\eqref{RAE}, condition \ref{T:KMK10AAP.d} additionally yields $\piopt\in L(R)$ via Theorem~\ref{T}. Hence in Theorem~\ref{T:KMK10AAP}, if $\piopt$ is chosen as $\lopt/p$ for a predictable local optimizer $\lopt$ of $\gfrak$, we may omit conditions~\ref{T:KMK10AAP.a} and \ref{T:KMK10AAP.c}, and the assumption that $\piopt\in L(R)$.

The predictable local optimizer $\lopt$ furnished in the context of this example by
Proposition~\ref{P} need not be $\mathcal H$-measurable.
The measurable-selection argument in
\citet[Theorem~3.1]{kallsen.muhle-karbe.10.aap}, however, yields a
predictable and $\mathcal H$-measurable process $\widetilde\pi$
satisfying the same pointwise optimality conditions. Hence
$p\widetilde\pi$, too, is a local optimizer of $\gfrak$.
Since $g$ is strictly concave and $g''(0)<0$,
Proposition~\ref{P}\ref{P.3} yields
$
    (\lopt-p\widetilde\pi)\sint R=0$.
Thus $\piopt=\lopt/p$ and the $\mathcal H$-measurable optimizer
$\widetilde\pi$ generate the same gains process.

\section{Counterexamples}\label{S:6}

All examples in this section illustrate tightness of Corollary~\ref{C} and Theorem~\ref{T} for deterministic time-constant local utility functions $g$.
The first example shows the importance of boundedness from above in Corollary~\ref{C}.
It shows that there exists an unbounded local utility function $g$ satisfying the  RAE condition~\eqref{RAE}, an instantaneously arbitrage-free semimartingale $R$, and a predictable process $\zeta\in L(R)$ such that $\gfrak<\infty$ but $\gfrak(\zeta)\sint A_T=\infty$. Indeed, $\zeta$ may be chosen to be the local optimizer.
\begin{example}\label{E2}
    Let $f: [0, \infty) \to [3, \infty)$ be a Borel-measurable function. Consider the one-dimensional time-inhomogeneous compound Poisson process $R$ with $A_t=t$, $b^{R[0]}=0$, $c^R=0$, and 
    \[F_t^R(\d x) = \left(\indicator{-1/2\leq x<0} + \indicator{\e-1<x\leq f(t)}\frac{1}{(1+x)(\ln(1+x))^2}\right)\d x, \qquad (t, x) \in [0, \infty) \times \R.\]
Such a process is well defined since $\int_\R\indicator{\abs{x}>1}F^R(\d x) \leq 1$, $\int_\R x^2\indicator{\abs{x}\leq1}F^R(\d x) = 1/24$, and $\abs{b^{\tr\circ R}} = 1/8$ all integrate $A$. Since $F^R$ charges both positive and negative values, $R$ is instantaneously arbitrage-free.

Consider now the unbounded local utility function $g=\ln(1+\id)$, which satisfies the RAE condition~\eqref{RAE}. Since $b^R$ is finite, we have $\gfrak<\infty$. Moreover,
\[
\edom\gfrak_t\subset\left[-\frac{1}{f(t)},2\right], \qquad t \geq 0.
\]
Indeed, for $t > 0$, if $\lambda>2$, then $1+\lambda x<0$ on a set of negative jumps of positive $F_t^R$-measure, while if $\lambda<-1/f(t)$, the same holds on a set of positive jumps of positive $F_t^R$-measure. Since $g''(0)=-1$ and $g$ satisfies the Inada condition~\eqref{eq:Inada},
Proposition~\ref{P}\ref{P.1} yields a predictable local optimizer $\lopt\in\edom\gfrak$; thus $\lopt\in L(R)$.

We have 
\[\gfrak_t(1)=\frac{1}{2}\ln2-\frac{1}{2}+\ln\left(\ln(1+f(t))\right)>0,\qquad t\geq0.\]
Choosing $f$ such that $\gfrak(1)$ is not integrable on $[0,1]$, e.g., $f(0)=3$ and $f(t)=\exp(\e^{1/t})+2$, $t>0$, gives
\[
\gfrak(\lopt)\sint A_1
\geq
\gfrak(1)\sint A_1
=
\infty.
\]
Thus the local optimizer integrates $R$, while its cumulative maximal local expected utility is infinite.
\end{example}

We now introduce two local utility functions that fail the RAE condition~\eqref{RAE} at different ends of the real line in order to illustrate why neither component of the RAE condition in Theorem~\ref{T} can simply be omitted. We summarise the main idea behind the construction of the examples here.

Inspecting the proof of Theorem~\ref{T}, the integrability of $\lopt$ with respect to $R$ stands and falls with the integrability of 
\[ b^{(\id^2\wedge\alpha^2)(\lopt\id),R}
+
\left|b^{(\id\indicator{\abs{\id}\leq\alpha})(\lopt\id),R}\right| \]
with respect to the activity process $A$, for some fixed $\alpha>0$.
 Finite cumulative utility together with~\eqref{eq:midQV} yields that $b^{(\id^2\wedge\alpha^2)(\lopt\id),R}$ must integrate $A$, hence the  integrability of $\lopt$ with respect to $R$ hinges on the $A$-integrability of $\abs{b^{\id\indicator{(\abs{\id}\leq\alpha})(\lopt\id),R}}$, which, by \eqref{eq:260519b} together with~\eqref{eq:260519d}, hinges on the $A$-integrability of the local expected utility generated by middling returns, $\abs{
b^{(g\indicator{\abs{\id}\leq\alpha})(\lopt\id),R}}$.
 
 In Example~\ref{E3}, the right-tail utility $
b^{(g\indicator{\id>\alpha})(\lopt\id),R}$ is $A$-integrable thanks to RAE at $+\infty$. Therefore, neither the left-tail utility $
b^{(g\indicator{\id<-\alpha})(\lopt\id),R}$ nor the midsection utility $b^{(g\indicator{\abs{\id}\leq\alpha})(\lopt\id),R}$ can be $A$-integrable if we want $\lopt\notin L(R)$, although their sum must integrate $A$ to make the cumulative local expected utility finite. At the same time, \eqref{eq:260519d} and $g-k\geq0$ dictate that $
b^{((g-k)\indicator{\id<-\alpha})(\lopt\id),R}$ must be $A$-integrable. Example~\ref{E3} shows that one can square these three objectives, fundamentally because with the RAE failing at $-\infty$, the gap $g-k$ is much smaller than $\abs{g}$ in the left tail. 

Example~\ref{E4} follows the same logic but with the RAE condition failing at $+\infty$, hence the lack of $A$-integrability now occurs in the midsection and the right tail of the utility.
 
\begin{example}\label{E3}
Consider the local utility function $g$ given by 
\begin{equation*}
g(x) = 
\begin{cases}
x(1+\ln(-x))-\frac{1}{2},       & x\leq-1;
\\[2ex]
x\wedge1 - \frac{1}{2}(x\wedge1)^2, & x > -1.
\end{cases}
\end{equation*}
This utility function is bounded above but fails the RAE condition~\eqref{RAE} in the left tail. Indeed, one has $$\lim_{x\downarrow-\infty}\frac{xg'(x)}{g(x)}=\lim_{x\downarrow-\infty}\frac{x(2+\ln(-x))}{x(1+\ln(-x))-\frac{1}{2}}=1.$$
However, $g$ satisfies the Inada condition~\eqref{eq:Inada} as well as $g''(0)<0$ and $g\leq 1/2$, hence by Proposition~\ref{P} a predictable local optimizer, here denoted $\lopt$, exists for any instantaneously arbitrage-free $R$. Furthermore, since $g$ is bounded above, integrability of $\lopt$ implies finite cumulative utility by Corollary~\ref{C}. 

Since $g$ does not satisfy the reasonable asymptotic elasticity condition at $-\infty$, we wish to understand whether, and if so how, $\lopt$ can have finite cumulative utility without integrability. 
To this end, we consider the sequence 
of probabilities $(p_n)_{n \in \N} \subset (0,1)$  given by 
\[p_n = \frac{1}{25(1 + (n \ln n)^2)}, \qquad n \in \N,\]
and define
\[
a_n
=
\frac{p_n n g'(-n)}{1-p_n}
=
\frac{n(2+\ln n)}
{25\bigl(1+(n\ln n)^2\bigr)-1}, \qquad n \in \N.
\]
This gives
$
a_1=1/12<1/4
$.
For \(n\ge2\), using \(\ln n\ge \frac12\), we obtain
\[
a_n
\le
\frac{1}{24}\frac{n(2+\ln n)}{n^2(\ln n)^2}
=
\frac{1}{24}\left(\frac{2}{n(\ln n)^2}+\frac{1}{n\ln n}\right)
\le
\frac{1}{24}(4+1)
=
\frac{5}{24}
<
\frac14.
\]
Using these probabilities, we consider next independent random variables $(X_n)_{n\in\N}$ such that 
$$\P\left[X_n=-n\right]=p_n,\qquad \P\left[X_n=u_n\right]=1-p_n,$$
where
\[ u_n = \frac{1-\sqrt{1-4a_n}}{2} = \frac{2a_n}{1+\sqrt{1-4a_n}} > 0.
\]
Observe, for $n\geq2$, we have
$u_n\geq a_n\geq 1/(50n\ln n)$, and hence
$\sum_{n\in\N}u_n=\infty$.

Since 
\[
\sum_{n \in \N} \frac{u_n}{n} <  
2 \sum_{n \in \N} \frac{a_n}{n} 
<2 \sum_{n \in \N} \frac{p_n g'(-n)}{1-p_n} < 12 \sum_{n \in \N} p_n n < \infty
\]
and since $\sum_{n \in \N} p_n < \infty$, Borel-Cantelli yields that
\[
    R_t = \sum_{n=1}^\infty \frac{X_n}{n} \indicator{\lc \tau_n, \infty\lc}, \qquad \text{where } \tau_n = 2-\frac{1}{n},
\]
is a finite-variation PII semimartingale $R$ equal to the sum of its jumps at the discrete times $(\tau_n)_{n\in\N}$. Note that $R$ is instantaneously arbitrage-free. We take
\[
A_t=\sum_{n\in\N}2^{-n}\indicator{\{\tau_n\leq t\}}, \qquad t \geq 0,
\]
so that
\[
F^R_{\tau_n}
=
2^n\left(
p_n\delta_{-1}
+
(1-p_n)\delta_{u_n/n}
\right),
\qquad n\in\N.
\]

Observe that
\[
    2^{-n}\gfrak_{\tau_n}(\lambda) = 
    p_n g(-\lambda) + (1-p_n) g\left(\lambda \frac{u_n}{n}\right), \qquad \lambda \in \R, n \in \N.
\]
Since
\begin{align*}
    2^{-n}\gfrak_{\tau_n}'(n) &= - p_n g'(-n) + (1-p_n) g'(u_n) \frac{u_n}{n}
    = - p_n g'(-n) + (1-p_n) (1 - u_n) \frac{u_n}{n}\\
    &= - p_n g'(-n)  + \frac{1-p_n}{n}  a_n
    = - p_n g'(-n)  + p_n g'(-n) = 0,
\end{align*}
where we used $u_n - u_n^2 = a_n$, we have $\lopt_{\tau_n}=n$ and  
\begin{align*}
    2^{-n}\gfrak_{\tau_n}(\lopt_{\tau_n}) &= p_n g(-n) + (1-p_n) g(u_n) 
    = p_n g(-n) + (1-p_n) \left(a_n + \frac{1}{2} u_n^2\right)\\
    &= p_n (g(-n) + n g'(-n)) + \frac{1}{2} (1-p_n) u_n^2
    = p_n \left(n - \frac{1}{2}\right) + \frac{1}{2} (1-p_n) u_n^2,
\end{align*}
for each $n\in\N$. Since $\sum_{n \in \N} p_n n < \infty$ and $\sum_{n \in \N} u_n^2 \leq 4\sum_{n \in \N} a_n^2 < \infty$, we have
\[
\gfrak(\lopt)\sint A_2
=
\sum_{n\in\N}2^{-n}\gfrak_{\tau_n}(\lopt_{\tau_n})<\infty,
\]
hence Theorem~\ref{T}\ref{T.i} is satisfied. However, $\lopt$ does not integrate $R$ since $\lopt \sint R_{\tau_n} = \sum_{k=1}^n X_k$ converges to $\infty$ as $n$ tends to infinity by Borel-Cantelli, thus Theorem~\ref{T}\ref{T.ii} does not hold.

One may wonder at this point whether our example is very degenerate in the sense of allowing arbitrage opportunities over well-behaved strategies. For each $n \in \N$, set
$r_n=u_n/(n+u_n)$ and let $L_n=r_n/p_n$ on $\{X_n=-n\}$ and
$L_n=(1-r_n)/(1-p_n)$ otherwise. Then $\E[L_n]=1$, the mean of $X_n$
under the probability measure with density $L_n$ is zero, and
\[
\E[L_n^2]
=
1+\frac{(r_n-p_n)^2}{p_n(1-p_n)}, \qquad n \in \N.
\]
Using $r_n\leq u_n/n$, $u_n<2a_n$, and
$a_n\leq 1/(8n\ln n)$ for $n\geq2$, we obtain
\[
\frac{r_n^2}{p_n}
\leq
\frac{25\left(1+(n\ln n)^2\right)}{16n^4(\ln n)^2},
\qquad n\geq2.
\]
Together with $\sum_{n\in\N}p_n<\infty$, this yields
$\sum_{n\in\N}(\E[L_n^2]-1)<\infty$. Independence therefore implies
that $Z_N=\prod_{n=1}^N L_n$ is an $L^2$-bounded density martingale and
converges in $L^2$ to some $Z$ with $\E[Z]=1$. Moreover,
$\sum_{n\in\N}(p_n+r_n)<\infty$ implies $Z>0$ almost surely, so $Z$
defines an equivalent probability measure $\Q$. By construction, the
jumps of $R$ have zero $\Q$-mean, and
\[
\sum_{n\in\N}\E^\Q[|\Delta R_{\tau_n}|]
=
\sum_{n\in\N}\frac{2u_n}{n+u_n}
<\infty.
\]
Thus $R$ admits an equivalent martingale measure with square-integrable density and, in particular, the model is arbitrage-free over all strategies whose maximal loss is square integrable.
\end{example}

\begin{example}\label{E4}
Consider the local utility function
\begin{equation*}
g(x) =
\begin{cases}
x - \frac{1}{2}x^2, & x\leq\frac{1}{2};
\\[2ex]
\frac{3}{8} + \frac{x-\frac{1}{2}}{2 (1+\ln(x+\frac{1}{2}))}, & x > \frac{1}{2}
\end{cases}
\end{equation*}
with derivative
\[
g'(x)=
\begin{cases}
1-x, & x\leq\frac12,\\[2ex]
\displaystyle
\frac{\ln\left(x+\frac12\right)+\frac{1}{x+\frac12}}
{2\left(1+\ln\left(x+\frac12\right)\right)^2},
& x>\frac12.
\end{cases}
\]
Then
\[
\lim_{x\uparrow\infty}\frac{xg'(x)}{g(x)}=1,
\qquad
\lim_{x\downarrow-\infty}\frac{xg'(x)}{g(x)}=2.
\]
Moreover, \(g\) satisfies the Inada condition~\eqref{eq:Inada}, \(g''(0)<0\), and \(g\) is unbounded above.

Similarly to Example~\ref{E3}, we now construct a process $R$ such that \(\lopt\) exists and has finite cumulative utility but does not integrate \(R\). To this end, for each $n \in \N$, consider
$
p_n=\frac{1}{n^2+1} \in (0,1)$
and define
\[
a_n
=
\frac{p_n n g'(n)}{1-p_n}
=
\frac{g'(n)}{n} > 0, \qquad
d_n
=
\frac{1-\sqrt{1+4a_n}}{2}
=
-\frac{2a_n}{1+\sqrt{1+4a_n}} < 0.
\]
Then 
\[
d_n-d_n^2=-a_n, \qquad n \in \N, \qquad \text{and} \qquad
\sum_{n\in\N}\frac{|d_n|}{n}<\infty,
\qquad
\sum_{n\in\N}|d_n|=\infty.
\]
Let now \((X_n)_{n\in\N}\) be independent random variables such that
\[
\P[X_n=n]=p_n,
\qquad
\P[X_n=d_n]=1-p_n.
\]
Since \(\sum_{n\in\N}p_n<\infty\), Borel--Cantelli implies that only finitely many events \(\{X_n=n\}\) occur. Hence, as previously, the process
\[
    R_t = \sum_{n\in\N} \frac{X_n}{n} \indicator{\lc \tau_n, \infty\lc}, \qquad \text{where } \tau_n = 2-\frac{1}{n},
\]
is a finite-variation PII semimartingale $R$ equal to the sum of its jumps at the discrete times $(\tau_n)_{n\in\N}$. 
Moreover, \(R\) is instantaneously arbitrage-free, since each jump distribution has support on both sides of zero. As in Example~\ref{E3}, we take
\[
A_t=\sum_{n\in\N}2^{-n}\indicator{\{\tau_n\leq t\}}, \qquad t \geq 0.
\]

The local expected utility now satisfies
\[
2^{-n}\gfrak_{\tau_n}(\lambda)
=
p_n g(\lambda)
+
(1-p_n)g\left(\lambda\frac{d_n}{n}\right),
\qquad \lambda\in\R, n \in \N,
\]
and 
\begin{align*}
2^{-n} \gfrak_{\tau_n}'(n)
={}&
p_n g'(n)
+
(1-p_n)g'(d_n)\frac{d_n}{n} 
=
p_n g'(n)
-
\frac{1-p_n}{n}a_n
=
0, \qquad n \in \N. 
\end{align*}
By concavity of \(\gfrak_{\tau_n}\) we hence have
$
\lopt_{\tau_n}=n$ for each $n \in \N$. 
Furthermore, we get
\begin{align*}
2^{-n}\gfrak_{\tau_n}(\lopt_{\tau_n})
&=
p_n g(n)+(1-p_n)\left(d_n-\frac12d_n^2\right) =
p_n g(n)-(1-p_n)a_n+\frac12(1-p_n)d_n^2 \\
&=
p_n\bigl(g(n)-n g'(n)\bigr)
+
\frac12(1-p_n)d_n^2\\
&\leq \frac{1}{n^2}
\left(
\frac{3}{8} + 
\frac{n \left(1+\ln\left(n+\frac12\right)\right)
- n \ln(n+\frac{1}{2})}
{2\left(1+\ln\left(n+\frac12\right)\right)^2}
\right) + d_n^2\\
&\leq \frac{1}{n^2}
+ 
\frac{1}
{2n\left(1+\ln\left(n+\frac12\right)\right)^2}
 + 4 a_n^2
, \qquad n \in \N.
\end{align*}
Hence we have $\gfrak(\lopt)\sint A_2
=
\sum_{n\in\N} 2^{-n}\gfrak_{\tau_n}(\lopt_{\tau_n})<\infty$,
so the cumulative local expected utility of the optimizer is finite.
On the other hand,
$
\lim_{n \uparrow \infty} \lopt\sint R_{\tau_n}
=
\sum_{n \in \N} X_n = -\infty$ by Borel--Cantelli, 
and thus \(\lopt\) does not integrate \(R\).
\end{example}

Although local expected utility typically does not admit a maximizer for continuous processes when $g''(0)=0$, such optimization remains well defined in discrete-time models (i.e., when $R$ is the sum of its jumps at predictable times in the \'Emery topology; see \citet{cerny.ruf.21.bej}). Here, for $\gfrak<\infty$, the local maximizers are guaranteed to exist even with $g''(0)=0$ under instantaneous absence of arbitrage~\eqref{eq:INA} and the Inada condition~\eqref{eq:Inada}; in fact, under the milder condition 
\begin{equation}\label{Inada2}
\frac{g'(\infty)
}{g'(-\infty)}=0,
\end{equation}
with the convention $0/0 = 0$. This was previously shown in 
\citet[Theorem~14]{cerny.al.12} for the special case of bounded jumps in a one-period setting. We next argue the general discrete-time case.

\begin{proposition}\label{P:discrete time}
Suppose $R$ equals the sum of its jumps at predictable times in the \'Emery topology. If $R$ is instantaneously arbitrage-free and $g$ satisfies the weaker Inada condition~\eqref{Inada2}, then $\gfrak$ admits a predictable optimizer. 
\end{proposition}

\begin{proof}
Fix a predictable time $\tau$. Without loss of generality it now suffices to argue that $\gfrak_\tau$ admits a predictable optimizer on
$\{\Delta A_\tau>0\}\cap\D$, and we assume throughout the proof that we are on this event. One then has
\[
b_\tau^{\tr\circ R}
=
\int_{\R^d}\tr(x)F^R_\tau(\d x),
\qquad
c^R_\tau=0
\]
by \cite[Proposition~II.2.9]{js.03}, and instantaneous absence of
arbitrage amounts to \eqref{eq:INA_tau}. On $\{g_\tau'(0)=0\}$, concavity
yields $g_\tau\leq0$, and hence zero is a local optimizer. We therefore
work on $\{g_\tau'(0)>0\}$. In this case $g_\tau'(-\infty)\in(0,\infty]$, and
\eqref{Inada2} yields $g_\tau'(\infty)=0$ if $g_\tau'(-\infty)\in\R$.

Fix $\lambda\in\edom\gfrak_\tau$. Since we work on $\D$,
$\gfrak_\tau(\lambda)\in\R$, and Lemma~\ref{L:260422Granada}, applied
to the triplet associated with $\lambda\id$, yields $\int_{\R}g_\tau'(\infty)x\indicator{0<x}F_\tau^{\lambda\id,R}(\d x)<\infty$ and
\[
\lim_{y\uparrow\infty}\frac{\gfrak_\tau(\lambda y)}{y}
=
\int_{\R}g_\tau'(\infty)x\indicator{0<x}F_\tau^{\lambda\id,R}(\d x)
+
\int_{\R}g_\tau'(-\infty)x\indicator{x<0}F_\tau^{\lambda\id,R}(\d x).
\]
If $F_\tau^{\lambda\id,R}(\R_-) >0$, the right-hand side is strictly
negative. Otherwise, \eqref{eq:INA_tau} yields
$F_\tau^{\lambda\id,R}(\R_+)=0$, and hence
$\gfrak_\tau(y\lambda)=0$ for all $y\in\R$, so that $\lambda$ belongs to the
lineality space of $\gfrak_\tau$. Thus the recession cone of $\gfrak_\tau$ is a linear subspace.
By Lemma~\ref{L:260428Granada}, $\gfrak_\tau$ is upper
semi-continuous on $\D$. Hence
\citet*[Theorem~27.3]{rockafellar.70} yields that $\gfrak_\tau$
attains its maximum on $\D$. The assertion now follows from
Lemma~\ref{L:260826London}.
\end{proof}

The next example shows that in discrete-time setting with $g''(0)=0$, one may encounter models with finite cumulative local expected utility in which the local optimizer does not integrate the underlying discrete-time stochastic process. Unlike Examples~\ref{E2}--\ref{E4} concerning the RAE condition~\eqref{RAE},  the next example only concerns local behaviour of $g$ near zero.

\begin{example}
For  $p\geq2$, consider the utility function
$$ g = \id \wedge 1 - \frac{1}{p}\abs{\id\wedge1}^p.$$
 Observe that for $p=2$ one has $g''(0)=-1$, while $g''(0)=0$ for all $p>2$.
Consider next a sequence of independent random variables $(X_n)_{n\in\N}$ such that 
$$\P\left[X_n=\pm \frac{1}{n+1}\right]=\frac{1}{2}\pm \frac{1}{2(n+1)^{(p-1)/2}}$$
and observe that $\E[X_n]=(n+1)^{-(p+1)/2}$ 
is summable. 
Letting
\[
    R_t = \sum_{n=1}^\infty X_n \indicator{\lc \tau_n, \infty\lc}, \qquad \text{where } \tau_n = 2-\frac{1}{n},
\]
 now yields a PII semimartingale $R$. 

Since $R$ is instantaneously arbitrage-free and $g'(-\infty)=\infty$, the local maximizer $\lopt$ exists by Proposition~\ref{P:discrete time}. Indeed, an easy calculation yields
$$\lopt_{\tau_n} = \sqrt{n+1}=\argmax_{\lambda\in\R}\E[g(\lambda\Delta R_{\tau_n})]
\qquad \text{with}
\qquad
\E[g(\lopt_{\tau_n}\Delta R_{\tau_n})]=\left(1-\frac{1}{p}\right)(n+1)^{-p/2},
\qquad
n \in \N.
$$ 
We also have $\E[(\lopt_{\tau_n}\Delta R_{\tau_n})^2]=(n+1)^{-1}$ for all $n\in\N$. 

For $p>2$, we therefore get:
\begin{enumerate}[(1)]
\item\label{e.i} The cumulative local expected utility is finite, i.e.,
$$ \int_0^2 \gfrak_t(\lopt_t) \d A_t = \sum_{n\in\N}\E[g(\lopt_{\tau_n}\Delta R_{\tau_n})]<\infty.$$
\item\label{e.ii} The process $\lopt$ does not integrate $R$ since, per Proposition~\ref{P:int},
\[
b^{(\id^2\wedge1)(\lopt\id),R}\sint A_2
=
\sum_{n\in\N}\E[(\lopt_{\tau_n}\Delta R_{\tau_n})^2]
=
\sum_{n\in\N}\frac{1}{n+1}
=
\infty.
\]
\end{enumerate} 
On the other hand, Theorem~\ref{T:CRS} asserts that \ref{e.i} and \ref{e.ii} cannot hold simultaneously for $p=2$.
\end{example}

\appendix
\section{Predictable measurability}\label{Appx:A}

\begin{lemma}\label{L:260826London}
Let $\ffrak:\R^d\to[-\infty,\infty]$ be a predictable function and set
\[
    \D^\ffrak=\{(\omega,t):\ffrak_t(\lambda)<\infty
    \text{ for all }\lambda\in\R^d\}.
\]
Then $\D^\ffrak$ admits a predictable version. Moreover, if $\ffrak$ attains
its maximum $(\P\times A)$-almost everywhere on $\D^\ffrak$, then there is an
$\R^d$-valued predictable process $\lopt$ such that
\[
    \ffrak(\lopt)=\sup_{\lambda\in\R^d}\ffrak(\lambda).
\]
\end{lemma}

\begin{proof}
If $\P\times A$ is the zero measure, there is nothing to prove.
Hence suppose that $\P\times A$ is non-zero.
Let $\mathcal{P}$ denote the predictable $\sigma$-field on
$\Omega\times[0,\infty)$. Since $\P\times A$ is $\sigma$-finite, there exists an
equivalent probability measure $\nu$ on $\mathcal{P}$. Moreover, since $\ffrak$ is predictable as a function
on $\Omega\times[0,\infty)\times\R^d$, the set
\[
 G=\left\{((\omega,t),\lambda):\ffrak_t(\lambda)=\infty\right\}
\]
belongs to $\mathcal{P}\otimes\mathcal{B}(\R^d)$,
and its projection
onto $\Omega\times[0,\infty)$ is $\D^{\ffrak,c}$.

By \citet*[Theorem~6.9.12]{bogachev.07}, $\D^{\ffrak,c}$ is a Souslin
set generated by $\mathcal{P}$. By
\citet*[Theorem~1.10.5]{bogachev.07}, such sets belong to $\mathcal{P}_\nu$, the
$\nu$-completion of $\mathcal{P}$. Since $\nu$ and
$\P\times A$ have the same null sets, $\D^\ffrak$ therefore admits
a predictable version.

For the second assertion, define the $[-\infty,\infty]$--valued process
\[
 X=\sup_{\lambda\in\R^d}\ffrak(\lambda)
\]
and note
\[
 \{X>a\}
 =
 \operatorname{proj}_{\Omega\times[0,\infty)}
 \left\{((\omega,t),\lambda):\ffrak_t(\lambda)>a\right\},
 \qquad a\in\R.
\]
The same argument as above yields that $X$ is
$\mathcal{P}_\nu$-measurable. By assumption, there exists
$N\in\mathcal{P}_\nu$ with $\nu(N)=0$ such that $\ffrak$ attains its
maximum on $\D^\ffrak\setminus N$. On $\D^{\ffrak,c}$, the maximum is attained as well,
since $\ffrak_t(\lambda)=\infty$ for some $\lambda\in\R^d$. Hence
\[
 H=
 \left\{((\omega,t),\lambda):
 \ffrak_t(\lambda)=X_t\right\}
 \cup(N\times\{0\})
\]
belongs to $\mathcal{P}_\nu\otimes\mathcal{B}(\R^d)$ and has nonempty
sections.

Since $(\Omega\times[0,\infty),\mathcal{P}_\nu,\nu)$ is a
complete probability space, the Aumann measurable selection theorem
\citep*[Theorem~6.9.13]{bogachev.07} yields an $\R^d$-valued
$\mathcal{P}_\nu$-measurable process whose graph is
contained in $H$. Since $\mathcal{P}_\nu$ is the $\nu$-completion of $\mathcal{P}$, this process admits a $\mathcal{P}$-measurable version $\lopt$, completing the proof.
\end{proof}

\section{Second-order Peano differentiability}\label{Appx:B}

\begin{proposition}[{\citet[Corollary~2.4, Theorem~2.8]{rockafellar.00}}] \label{P:rock}
A concave function $g:\R\to[-\infty, \infty]$ is twice Peano-differentiable at $0$ if and only if $g'(0)$ exists and there is $A\leq0$ such that 
\begin{align*}
    \lim_{x\to0}\frac{g'_+(x)-g'(0)-Ax}{\abs{x}} = \lim_{x\to0}\frac{g'_-(x)-g'(0)-Ax}{\abs{x}} = 0,
\end{align*}
in which case $A=(g_+')'(0)=(g_-')'(0)$ is the classical second derivative $g''(0)$ of $g$ computed over the punctured neighbourhood of zero where $g'$ exists.    
\end{proposition}
\begin{lemma}\label{L:260508London}
    If a concave function $g:\R\to[-\infty, \infty]$ is twice Peano-differentiable at $0$, then for any $\lambda,\lambda'\in\R^d$, the function 
    \[g'_+(\lambda\id)(\lambda'\id)^+
    - g'_-(\lambda\id)(\lambda'\id)^-  \]
    is twice Peano differentiable at $0$ with derivative $g'(0) \lambda'$ and second-order Peano derivative 
    $g''(0) (\lambda^\top \lambda' + \lambda'^\top \lambda)$.
\end{lemma}
\begin{proof}
Fix $\lambda,\lambda'\in\R^d$. Twice Peano-differentiability of $g$ and Proposition~\ref{P:rock} yield
\begin{align*}
\lim_{x\to0}\frac{g'_+(\lambda x)-g'(0)-g''(0)\lambda x}{\abs{x}}=
    \lim_{x\to0}\frac{g'_-(\lambda x)-g'(0)-g''(0)\lambda x}{\abs{x}}  = 0,
\end{align*}
which in turn gives
\begin{align*}
\lim_{x\to0}\frac{g'_+(\lambda x)-g'(0)-g''(0)\lambda x}{\abs{x}}\frac{(\lambda'x)^+}{\abs{x}} 
=
    \lim_{x\to0}\frac{g'_-(\lambda x)-g'(0)-g''(0)\lambda x}{\abs{x}} \frac{(\lambda'x)^-}{\abs{x}}=  0.
\end{align*}
Letting $f= g'_+(\lambda\id)(\lambda'\id)^+ - g'_-(\lambda\id)(\lambda'\id)^-$, this then yields
\begin{align*}
    \lim_{x\to0}\frac{f-g'(0)\lambda' x-g''(0)(\lambda x)(\lambda' x)}{\abs{x}^2} = 0,
\end{align*}
which shows that $f$ is twice Peano-differentiable at $0$.
\end{proof}

\section{Calculus on L\'evy triplets and related optimization problems}\label{Appx:C}

Subsection~\ref{SS:C.1} of this appendix develops the calculus of universal V-functions acting on L\'evy triplets. The calculus is purely analytic in that it has no $(\omega,t)$ dimension. However, if one were to associate with a given triplet a hypothetical L\'evy process $L$, the action of a compatible universal  V-function $\xi$ on the original triplet might be interpreted as producing the triplet of the corresponding process-level variation $\xi\circ L$. 

Subsection~\ref{SS:C.2} formulates properties of local expected utility reduced to this analytic setting. 

\subsection{Calculus on L\'evy triplets}\label{SS:C.1}

\begin{definition}
Let $m,n\in\N$.
\begin{enumerate}[(1)]
\item
A measure $F$ on $\R^m$ that integrates $\abs{\id}^2\wedge1$ and
satisfies $F(\{0\})=0$ is called a L\'evy measure.

\item
We say that $(b,c,F)$ is an $\R^m$-triplet if $b\in\R^m$,
$c\in\R^{m\times m}$ is positive semi-definite, and $F$ is a
L\'evy measure on $\R^m$.

\item
We say that a Borel-measurable function
$\xi:\mathcal U\to\R^n$ is a universal V-function if
$\mathcal U\subset\R^m$ is Borel and contains an open neighbourhood
of zero, $\xi(0)=0$, and $\xi$ is twice Peano-differentiable at zero. 

\item
We say that a universal V-function $\xi$  is $F$-compatible if
$F((\dom\xi)^c)=0$.

\item
We say that a universal V-function $\xi$ is 
$F$-special if it is $F$-compatible and 
$
\abs{\xi}\indicator{\abs{\xi}>1}
$
integrates $F$.
\end{enumerate}
\end{definition}

From here onwards, we will call a `universal V-function' just V-function.

\begin{lemma}\label{L:B1} Let $F$ be a L\'evy measure on $\R^m$ and $\xi$ an $F$-compatible  V-function. Then the following statements hold.
\begin{enumerate}[(1)]
\item\label{L:B1.1} If $D\xi(0)=0$ and $\xi$ is bounded (on $\dom\xi$), then $\xi$ integrates $F$.
\item\label{L:B1.2} If $D\xi(0)=0$, then $\xi$ integrates $F$ if and only if $\abs{\xi}\indicator{\abs{\xi}>K}$ integrates $F$ for some (equivalently any) $K>0$.
\item\label{L:B1.3} The image measure of $F$ via $\xi$, with its mass at zero removed, is a L\'evy measure.
\end{enumerate}
\end{lemma}

\begin{proof}
For \ref{L:B1.1} and \ref{L:B1.2}, suppose that $D\xi(0)=0$. Since $\xi$ is twice Peano-differentiable at zero, there exist $\varepsilon>0$ and $C<\infty$ such that
\[
\abs{\xi(x)}\leq C\abs{x}^2,\qquad \abs{x}\leq\varepsilon.
\]
Moreover, since $F$ is a L\'evy measure,
\[
\int_{\abs{x}\leq\varepsilon}\abs{x}^2F(\d x)<\infty
\qquad\text{and}\qquad
F(\{\abs{x}>\varepsilon\})<\infty.
\]

Note that \ref{L:B1.2} implies \ref{L:B1.1}.
For \ref{L:B1.2}, fix $K>0$. Since $\xi$ is continuous at zero, after decreasing $\varepsilon$ if necessary we may assume that $\abs{\xi(x)}\leq K$ whenever $\abs{x}\leq\varepsilon$. Hence
\[
\int_{\R^m}
\abs{\xi(x)}
\indicator{\abs{\xi(x)}\leq K}
F(\d x)
\leq
C\int_{\abs{x}\leq\varepsilon}\abs{x}^2F(\d x)
+
K F(\{\abs{x}>\varepsilon\})
<\infty.
\]
Thus $\xi$ integrates $F$ if and only if
$\abs{\xi}\indicator{\abs{\xi}>K}$ integrates $F$. Since this holds for every $K>0$, ``some'' and ``any'' are equivalent.

For \ref{L:B1.3}, differentiability of $\xi$ at zero and $\xi(0)=0$ yield, possibly after changing $\varepsilon$ and $C$,
\[
\abs{\xi(x)}\leq C\abs{x},\qquad \abs{x}\leq\varepsilon.
\]
Let $F^\xi$ denote the image measure of $F$ via $\xi$, with mass at zero removed. Then $F^\xi(\{0\})=0$ and
\[
\int_{\R^n}(\abs{y}^2\wedge1)F^\xi(\d y)
=
\int_{\R^m}(\abs{\xi(x)}^2\wedge1)F(\d x)
\leq
C^2\int_{\abs{x}\leq\varepsilon}\abs{x}^2F(\d x)
+
F(\{\abs{x}>\varepsilon\})
<\infty,
\]
and hence $F^\xi$ is a L\'evy measure.
\end{proof}

\begin{definition} Let $m,n\in\N$ and fix an $\R^m$-triplet $(b^{\tr},c,F)$. 
\begin{enumerate}[(1)]
\item
If $\xi$ is a V-function on $\R^m$ with values in $\R^n$, we let 
\begin{align*} 
c^\xi&{} = D\xi(0)c(D\xi(0))^\top,\\
F^\xi &{}\text{ is the image measure of $F$ via $\xi$  with its mass at zero removed.} 
\end{align*}
If, furthermore, $\xi$ is $F$-special, we let
\begin{equation}\label{eq:bxi}
b^{\xi} = D\xi(0)b^{\tr}+\frac{1}{2}\trace(D^2\xi(0)c)+\int_{\R^m} \bigs(\xi(x)-D\xi(0) \tr(x)\bigs)F(\d x),
\end{equation}
which is finite by Lemma~\ref{L:B1}\ref{L:B1.2} applied to $\xi-D\xi(0)\tr$.
\item If $\xi$ is a V-function and compatible with $F$, we write 
\[(b^{\tr(\xi)},c^\xi,F^\xi) = (b^{\tr},c,F)^\xi.\]
\end{enumerate}
\end{definition}
 
Observe that the notation $b^{\tr}$ is consistent with \eqref{eq:bxi}: since $\tr$ in \eqref{tr} is twice Peano-differentiable, $D\tr(0)$ is an identity matrix and $D^2\tr(0)=0$. However, since $\tr$ in \eqref{tr} is twice Peano-differentiable, $D\tr(0)$ is an identity matrix, and $D^2\tr(0)=0$, formula~\eqref{eq:bxi} with $\xi=\tr$ is self-consistent.

\begin{lemma}\label{L:B2} Let $m,n\in\N$ and fix an $\R^m$-triplet $(b^{\tr},c,F)$.
\begin{enumerate}[(1)] 
\item\label{L:B2.1} If $\xi$ is a V-function on $\R^m$ with values in $\R^n$, and  $F((\dom \xi)^c)=0$, then 
$(b^{\tr(\xi)},c^\xi,F^\xi)$ is an $\R^n$-triplet. 
\item\label{L:B2.2} If, additionally, $\psi$ is a V-function on $\R^n$, then the following are equivalent.    
   \begin{enumerate}[(i)]
     \item $F^\xi((\dom \psi)^c)=0$.
     \item $F((\dom \psi(\xi))^c)=0$.
   \end{enumerate}
Furthermore, if either of these conditions holds, then  
\[(b^{\tr(\xi)},c^\xi,F^\xi)^\psi = (b^{\tr},c,F)^{\psi(\xi)}. \]
If, additionally, $\abs{\psi(\xi)}\indicator{\abs{\psi(\xi)}>1}$ integrates $F$, then
\begin{equation} \label{eq:260920}
b^{\psi(\xi)} = D\psi(0)b^{\tr(\xi)}+\frac{1}{2}\trace(D^2\psi(0)c^\xi)+\int_{\R^n} \bigs(\psi(x)-D\psi(0) \tr(x)\bigs)F^\xi(\d x).
\end{equation}
\item\label{L:B2.3}
Let $\xi,\eta$ be real-valued V-functions on a common domain with
$F((\dom\xi)^c)=0$.
If $b^\xi$ and $b^\eta$ are defined, then, for $a_1,a_2\in\R$,
\[
b^{a_1\xi+a_2\eta}
=
a_1b^\xi+a_2b^\eta.
\]
If $\eta\geq0$ and $b^\eta$ is defined, then every real-valued
V-function $\xi$ satisfying
$
\abs{\xi}\leq\eta$
has $b^\xi$ defined and
\[
\abs{b^\xi}\leq b^\eta.
\]
\end{enumerate}
\end{lemma}
\begin{proof}
For \ref{L:B2.1}, the positive semidefiniteness of $c^\xi$ is evident. Since $\tr(\xi)$ is bounded on $\dom\xi$, $b^{\tr(\xi)}$ is finite by Lemma~\ref{L:B1}\ref{L:B1.2} (applied to $\tr(\xi)-D\xi(0)\tr$). The same lemma yields that $F^\xi$ is a L\'evy measure. Thus $(b^{\tr(\xi)},c^\xi,F^\xi)$ is an $\R^n$-triplet.

For \ref{L:B2.2}, the equivalence follows directly from the definition of $F^\xi$. Since $(b^{\tr(\xi)},c^\xi,F^\xi)$ is an $\R^n$-triplet and $F^\xi((\dom\psi)^c)=0$, the operation $(b^{\tr(\xi)},c^\xi,F^\xi)^\psi$ is well defined. Observe that $\eta=\psi(\xi)$ is twice Peano-differentiable at zero with 
\begin{align} 
D\eta(0) {}&=  D\psi(0)D\xi(0), \label{eq:Deta}\\
D^2\eta(0) {}&=   (D\xi(0))^\top  D^2 \psi(0)D\xi(0)  + \sum_{k = 1}^{n} D_k\psi(0)  D^2\xi^{(k)}(0). \label{eq:D2eta}
\end{align} 
The first component of $(b^{\tr(\xi)},c^\xi,F^\xi)^\psi$ reads 
\begin{align*}
    D\psi(0)&{}b^{\tr(\xi)}+\frac{1}{2}\trace(D^2\psi(0)c^\xi)+\int_{\R^n} \bigs(\tr(\psi(x))-D\psi(0) \tr(x)\bigs)F^\xi(\d x)\\
    ={}& D\psi(0)\left(D\xi(0)b^{\tr}+\frac{1}{2}\trace(D^2\xi(0)c)+\int_{\R^m} \bigs(\tr(\xi(x))-D\xi(0) \tr(x)\bigs)F(\d x)\right)\\
    &{}+\frac{1}{2}\trace(D^2\psi(0)D\xi(0)c(D\xi(0))^\top)\\
    &{}+\int_{\R^m} \bigs(\tr(\psi(\xi(x)))-D\psi(0) \tr(\xi(x))\bigs)F(\d x)\\
    ={}& D\eta(0)b^{\tr}+\frac12\trace(D^2\eta(0)c)+\int_{\R^m}\left(\tr(\eta(x))-D\eta(0)\tr(x)\right)F(\d x)\\
    ={}&b^{\tr(\eta)}.
\end{align*}  
Here we have used \eqref{eq:Deta}, \eqref{eq:D2eta}, $D\tr(\varphi)(0)=D\varphi(0)$ for $\varphi\in\{\xi,\psi,\eta\}$, $\trace(AB)=\trace(BA)$ for compatible matrices, and properties of image measures. The identities for the second and third component of $(b^{\tr(\xi)},c^\xi,F^\xi)^\psi$ follow easily. The identity in \eqref{eq:260920} follows from 
\[b^\eta-b^{\tr(\eta)}= \int_{\R^m}(\eta-\tr(\eta))F(\d x) = \int_{\R^n}(\psi-\tr(\psi))F^\xi(\d x).\]

For \ref{L:B2.3}, linearity follows directly from
\eqref{eq:bxi}. If $\xi\leq\eta$, then $\eta-\xi$ has a minimum at
zero. Hence
\[
D(\eta-\xi)(0)=0
\qquad\text{and}\qquad
D^2(\eta-\xi)(0)
\]
is positive semi-definite. Therefore, by \eqref{eq:bxi},
\[
b^\eta-b^\xi
=
\frac{1}{2}\trace\left(D^2(\eta-\xi)(0)c\right)
+
\int_{\R^m}(\eta-\xi)(x)F(\d x)
\geq0.
\]
Finally, if $\abs{\xi}\leq\eta$ and $b^\eta$ is defined, then
$\abs{\xi}\indicator{\abs{\xi}>1}$ integrates $F$, so $b^\xi$ is
defined. Applying the preceding inequality to
$\xi\leq\eta$ and $-\xi\leq\eta$ gives
$
\abs{b^\xi}\leq b^\eta$.
\end{proof}

\subsection{Existence of optimizers of certain concave functions}\label{SS:C.2}

Fix a concave non-decreasing V-function $g$ on $\R$. For $H>0$ and $K>0$, the V-function
\[g_{H,K}=(-H\vee g)\wedge K\] 
is bounded on $\R$. For an $\R^d$-triplet $(b^{\tr},c,F)$, define the concave function $\gfrak:\R^d\to[-\infty,\infty]$ by setting
\[\gfrak(\lambda) = \lim_{K\uparrow\infty}\lim_{H\uparrow\infty}
b^{g_{H,K}(\lambda\id)}. \]
Observe that explicitly one has 
\[
\gfrak(\lambda)
=
g'(0)b^{\tr}\lambda
+\frac{1}{2}g''(0)\lambda c\lambda^\top
+\int_{\R^d}
\bigl(g(\lambda x)-g'(0)\lambda\tr(x)\bigr)F(\d x),
\]
with the same extended integral convention as in \eqref{eq:locutil}.

\begin{lemma}\label{L:260428Granada}
If $g$ is upper semi-continuous and $\gfrak<\infty$, then  $\gfrak$ is upper semi-continuous.
\end{lemma}
\begin{proof}
Let $\widehat F$ be the restriction of $F$ to
$\{\abs{x}>1\}$. Observe that $\widehat F(\R^d)<\infty$.
This leads to a natural decomposition
$\gfrak=\widebar\gfrak+\widehat\gfrak$ with
\begin{align*}
\widebar\gfrak(\lambda)
={}&
g'(0)b^{\id\indicator{\abs{\id}\leq1}}\lambda
+\frac{1}{2}g''(0)\lambda c\lambda^\top
+\int_{\R^d}
\bigs(g(\lambda x)-g'(0)\lambda x\bigs)
\indicator{\abs{x}\leq1}F(\d x),\\
\widehat\gfrak(\lambda)
={}&
\int_{\R^d}g(\lambda x)\widehat F(\d x).
\end{align*}

Since $g-g'(0)\id\leq0$, Fatou's lemma yields upper
semi-continuity of $\widebar\gfrak$. Furthermore, if $g^-$ is
finite-valued, then $\dom\widebar\gfrak=\R^d$, and otherwise
$\lambda\in\dom\widebar\gfrak$ for all $\lambda$ with
$\abs{\lambda}<\abs{\inf\dom g}$. Hence the conic hull of
$\dom\widebar\gfrak$ spans the whole of $\R^d$ in either case.

We next show that $\widehat\gfrak$ is upper semi-continuous.
Denote by $C$ the cone generated by $\dom\gfrak$, which by the
previous argument coincides with the cone generated by
$\dom\widehat\gfrak$. For $n=\dim C<1$, there is nothing to prove
since then $\dom\widehat\gfrak=\{0\}$ is a singleton. For $n\geq1$,
there are $n=\dim C$ linearly independent vectors
$\lambda_1,\ldots,\lambda_n\in\dom\widehat\gfrak$ whose linear span
contains $\dom\widehat\gfrak$.

Since $\gfrak<\infty$, we have
\[
\int_{\R^d}g(\lambda_i x)^+\widehat F(\d x)<\infty.
\]
In view of $\lambda_i\in\dom\widehat\gfrak$ and
$g(-\id)^-\geq g(\id)^+$, this gives
\[
\int_{\R^d}g(\abs{\lambda_i x})\widehat F(\d x)<\infty,
\qquad i\in\{1,\ldots,n\}.
\]

For every bounded subset $A\subset\dom\widehat\gfrak$, there is
$\delta>0$ such that
\[
A\subset
\left\{
\sum_{i=1}^n\alpha_i\lambda_i:
\alpha\in\R^n,\abs{\alpha}<\delta
\right\}.
\]
By monotonicity and subadditivity of $g^+$, we have
\[
\sup_{\lambda\in A}
g\left(\sum_{i=1}^n\alpha_i\lambda_i\id\right)^+
\leq
\sum_{i=1}^n
\max(\abs{\alpha_i},1)g(\abs{\lambda_i\id})
\leq
\max(\delta,1)
\sum_{i=1}^n g(\abs{\lambda_i\id}).
\]
It follows that
$\{g(\lambda\id)^+:\lambda\in A\}$ admits an integrable upper
bound with respect to $\widehat F$. Thus every convergent sequence
in $\dom\widehat\gfrak$ admits an integrable upper bound, and the
reverse Fatou lemma yields that $\widehat\gfrak$ is upper
semi-continuous. Hence $\gfrak$ is upper semi-continuous.
\end{proof}

We next study the asymptotic slope of function $\gfrak$ in one dimension. We shall later apply Lemma~\ref{L:260422Granada} to the $\R$-triplet $(b^{\tr},c,F)^{\lambda\id}$ with $\lambda\in\R^d$.  
\begin{lemma}\label{L:260422Granada}
Let $(b^{\tr},c,F)$ be an $\R$-triplet. Assuming $\int g(x)\indicator{1<x}F(\d x) <\infty$, the function $\psi:[0,\infty)\to [-\infty,\infty)$ is well defined by 
\begin{equation}\label{eq:psi} 
\psi(y)= g'(0)b^{\tr}y + \frac{1}{2}g''(0)c y^2 +\int_\R \left(g(xy)-g'(0)\tr(x)y\right)F(\d x),\quad y\geq0,
\end{equation} 
$\int_\R g'(\infty) x\indicator{1<x}F(\d x) < \infty$, and
\begin{equation}\label{eq:230718}
\begin{split}
 \infty>\lim_{y\to\infty}\frac{\psi(y)}{y} ={}& g'(0)b^{\tr}-\infty\indicator{g''(0)c<0}\\
 &{}+\int_\R\left(g'(\infty)x-g'(0)\tr(x)\right)\indicator{0<x}F(\d x)\\
 &{}+\int_\R\left(g'(-\infty)x-g'(0)\tr(x)\right)\indicator{x<0}F(\d x).
\end{split}
\end{equation}
\end{lemma}

\begin{proof}
Since $g$ is non-decreasing, concave, and zero at $0$, we have 
\begin{equation*}
0\leq g(xy)\indicator{1<x}\leq (y\vee1) g(x)\indicator{1<x},\qquad y\geq 0, 
\end{equation*}
which yields 
\[
\int_{\R}g(xy)\indicator{1<x}F(\d x)<\infty,\qquad y\geq0.
\]
Moreover, concavity gives 
\[
g(xy)-g'(0)\tr(x)y\leq0,\qquad x\leq1,\quad y\geq0,
\]
and hence $\psi<\infty$ is well defined by \eqref{eq:psi} on $[0,\infty)$.  

Concavity and monotonicity further yield $g(y)/y\downarrow g'(\infty)$ as $y\uparrow\infty$. Thus, for $y\geq1$,
\[
0\leq y^{-1}g(xy)\indicator{1<x}\leq g(x)\indicator{1<x},
\qquad
\lim_{y\uparrow\infty}
y^{-1}g(xy)\indicator{1<x}
= g'(\infty)x\indicator{1<x},
\]
and dominated convergence gives
\[
\int_\R g'(\infty) x\indicator{1<x}F(\d x)<\infty.
\]
Furthermore,
\[
0\geq y^{-1}\left(g(xy)-g'(0)\tr(x)y\right)\indicator{0<x\leq1}
\downarrow
\left(g'(\infty)x-g'(0)\tr(x)\right)\indicator{0<x\leq1},
\qquad y\uparrow\infty,
\]
and, similarly, $g(y)/y\uparrow g'(-\infty)$ as $y\downarrow-\infty$ gives
\[
0\geq y^{-1}\left(g(xy)-g'(0)\tr(x)y\right)\indicator{x<0}
\downarrow
\left(g'(-\infty)x-g'(0)\tr(x)\right)\indicator{x<0},
\qquad y\uparrow\infty.
\]
An application of monotone convergence to the last two displays, together with the preceding dominated convergence 
yields \eqref{eq:230718}.
\end{proof}

We next note existence, and in certain circumstances also uniqueness, of a maximizer.  

\begin{lemma}\label{L:260830London}
Assume $g''(0)<0$. Then the following statements hold.
\begin{enumerate}[(1)]
\item\label{London.1} The lineality space of $\gfrak$, i.e., the largest linear subspace inside the recession cone of $\gfrak$, equals
\[
\mathcal{L}
=
\left\{
\lambda\in\R^d:
\lambda b^{\tr}
-
\int_{\R^d}\lambda \tr(x)F(\d x)=0,\quad
\lambda c=0,\quad
\int_{\R^d}\indicator{\lambda x\ne0}F(\d x)=0
\right\}.
\]
\item\label{London.2} If $g$ is upper semi-continuous and satisfies the Inada condition~\eqref{eq:Inada}, and
$(b^{\tr},c,F)$ is instantaneously arbitrage-free, then $\gfrak$ attains its maximum in $\R^d$.
\item\label{London.3} If $\gfrak<\infty$ and $g$ is strictly concave, then any
two maximizers of $\gfrak$ differ by an element of $\mathcal{L}$.
\end{enumerate}
\end{lemma}

\begin{proof}
Observe first that $g''(0)<0$ yields $g'(0)>0$ by monotonicity of $g$.

Let us first argue \ref{London.1}.
Suppose that $\lambda$ belongs to the lineality space of $\gfrak$. Then
$
    \gfrak(\lambda)=0=\gfrak(-\lambda)$.
Consequently,
\[
0
=
\gfrak(\lambda)+\gfrak(-\lambda)
=
g''(0)\lambda c\lambda^\top
+
\int_{\R^d}
\left(
g(\lambda x)+g(-\lambda x)
\right)F(\d x).
\]
Both terms on the right-hand side are nonpositive and therefore equal
to zero. This yields $\lambda c=0$. Moreover, concavity and $g(0)=0$
give
$
g(z)+g(-z)\leq0$ for all $z\in\R$, 
with strict inequality for $z\ne0$. Hence
$
\int_{\R^d}\indicator{\lambda x\ne0}F(\d x)=0$.
Substitution into $\gfrak(\lambda)=0$ gives
\[
\lambda b^{\tr}
-
\int_{\R^d}\lambda\tr(x)F(\d x)=0,
\]
and hence $\lambda\in\mathcal L$.

Conversely, if $\lambda\in\mathcal L$, a direct substitution yields
$\gfrak(y\lambda)=\gfrak(0)$ for all $y \in \R$.
Thus $\lambda$ belongs to the lineality space of $\gfrak$, which proves
\ref{London.1}.

Suppose now that the assumptions of \ref{London.2} hold in addition. If $\gfrak\not<\infty$, then $\gfrak(\lambda)=\infty$ for some
$\lambda\in\R^d$, and there is nothing to prove. Hence suppose
$\gfrak<\infty$.
Fix $\lambda\in\edom\gfrak$.
The condition $\gfrak<\infty$ yields
$g\indicator{\id>1}$ integrates $F^{\lambda\id}$, hence
$0\leq g(y\id)\indicator{\id>1}\leq (y\vee1)g\indicator{\id>1}$
integrates $F^{\lambda\id}$ for all $y\geq0$.
Using
\begin{equation}\label{260901Santiago}
\left.
\begin{aligned}
b^{\tr(\lambda \id)}
 &{}=
 b^{\tr}\lambda
 +\int_{\R^d}
 \bigs(\tr(
   \lambda x)
   -\lambda \tr(x)
 \bigs)F(\d x),\\
c^{\lambda\id}
 &{}=\lambda c\lambda^\top,\\
F^{\lambda\id}
 &{}\text{ is the image measure of $F$ under
 $\lambda\id$, with its mass at zero removed\qquad}
 \end{aligned}\right\}
\end{equation}
one then obtains from Lemma~\ref{L:B2}\ref{L:B2.2}, with
$\xi=\lambda\id$ and $\psi = g_{H,K}(y\id)$,
\[
\gfrak(\lambda y)
=
g'(0)b^{\tr(\lambda\id)} y
+\frac12g''(0)c^{\lambda \id}y^2
+\int_{\R}
\bigs(
  g(xy)-g'(0)\tr(x)y
\bigs)F^{\lambda\id}(\d x),
\qquad y\geq0.
\]
The Inada condition~\eqref{eq:Inada} and
Lemma~\ref{L:260422Granada} therefore yield
\begin{equation}\label{eq:260424Granada}
\begin{split}
\infty>
\lim_{y\uparrow\infty}\frac{\gfrak(\lambda y)}{y}
={}&
g'(0)
\left(
 b^{\tr(\lambda\id)}
 -
 \int_{\R}\tr(x)\indicator{0<x}F^{\lambda\id}(\d x)
\right)
-
\infty
\left(
 \indicator{c^{\lambda\id}>0}
 +
 \indicator{F^{\lambda\id}(\R_{-})>0}
\right).
\end{split}
\end{equation}

If $c^{\lambda \id}>0$ or
$F^{\lambda \id}(\R_-)>0$, the limit in
\eqref{eq:260424Granada} equals $-\infty$. In the remaining case,
\begin{equation*}
 \text{$c^{\lambda \id}=0$ and
$F^{\lambda \id}(\R_-)=0$.}
\end{equation*}
In view of \eqref{260901Santiago},
instantaneous absence of arbitrage now gives one of the following two
subcases. Either
\[
F^{\lambda\id}(\R_+)>0
\quad\text{and}\quad
b^{\tr(\lambda\id)}
-
\int_{\R}\tr(x)
F^{\lambda\id}(\d x)
<0,
\]
in which case the limit in \eqref{eq:260424Granada} is strictly
negative, or
\begin{align*} 
F^{\lambda\id}(\R_+)=0
\quad\text{and}\quad
b^{\tr(\lambda\id)}
-
\int_{\R}\tr(x)
F^{\lambda\id}(\d x)
=
b^{\tr(\lambda\id)}
=0.
\end{align*}
In the latter subcase,
$
    \gfrak(\lambda y)=0$ for all $y\in\R$.
Thus $\lambda$ belongs to the lineality space of $\gfrak$.

Therefore, for every $\lambda\in\dom\gfrak$, either $\lambda$ belongs
to the lineality space of $\gfrak$, or
\[
    \lim_{y\uparrow\infty}
    \frac{\gfrak(\lambda y)}{y}<0,
\]
and hence $\gfrak(\lambda y)\to-\infty$ as
$y\uparrow\infty$. By Lemma~\ref{L:260428Granada}, $\gfrak$ is upper
semi-continuous. It now follows from
\citet*[Theorem~27.3]{rockafellar.70} that $\gfrak$ attains its
maximum, yielding \ref{London.2}.

Finally, suppose that the assumptions of \ref{London.3} hold and let
$\lambda,\lambda'$ be two maximizers of $\gfrak$. Set
$
    \zeta=\lambda-\lambda'.
$
Since $\gfrak$ is concave, $(\lambda+\lambda')/2$ too is a
maximizer, and therefore
\begin{align*}
0
={}&
\gfrak\left(\frac{\lambda+\lambda'}{2}\right)
-\frac12\gfrak(\lambda)
-\frac12\gfrak(\lambda')\\
={}&
-\frac18g''(0)\zeta c\zeta^\top
+
\int_{\R^d}
\left(
g\left(\frac{\lambda+\lambda'}{2}x\right)
-\frac12g(\lambda x)
-\frac12g(\lambda'x)
\right)F(\d x).
\end{align*}
Both terms on the right-hand side are nonnegative and therefore equal
to zero. This yields $\zeta c=0$ and, by strict concavity of $g$,
$\zeta\id=0$, $F$-almost everywhere. Since $\gfrak(\lambda)=\gfrak(\lambda')\in\R $, 
\[
0
=
\gfrak(\lambda)-\gfrak(\lambda')
=
g'(0)
\left(
 \zeta b^{\tr}
 -
 \int_{\R^d}\zeta\tr(x)F(\d x)
\right).
\]
Hence $\zeta\in\mathcal L$, yielding \ref{London.3}.
\end{proof}

Let $f:\R^d\to[-\infty,\infty]$ be a concave function. Recall that for $x\in\R^d$ with $f(x)$ finite and $y\in\R^d$, the directional derivative 
\[ 
f'(x;y)
=
\lim_{\varepsilon\downarrow0}
\frac{
f\bigl(x+\varepsilon y\bigr)
-f(x)
}{\varepsilon}
\]
is well defined by concavity with values in $[-\infty,\infty]$ and, furthermore,
\[f'(x;y)\geq
f\bigl(x+y\bigr)
-f(x);\] see \citet[Theorem~23.1]{rockafellar.70}.

The next lemma gives a differential characterisation of an optimizer of $\gfrak$.
\begin{lemma}\label{L:FOC Levy}
Fix $\lopt\in\R^d$. For each $\lambda\in\R^d$, 
\begin{align} \label{eq:260925.1}
    \xi^\lambda = g'_+(\lopt\id)((\lambda-\lopt)\id)^+-g'_-(\lopt\id)((\lambda-\lopt)\id)^-
\end{align}
is a V-function.
Moreover,
    $\lopt$ maximizes $\gfrak$ if and only if $\gfrak(\lopt)=\infty$ or $\gfrak(\lopt)\in\R$ and $\xi^\lambda$
is $F$-special for all $\lambda\in\edom\gfrak$, and  
\begin{align*} 
-\infty<b^{\xi^\lambda}\leq 0,\quad \lambda\in\edom\gfrak.
\end{align*}
 In the latter case, $\gfrak<\infty$ and the V-function
$-\xi^0=(g'_-\id^+-g'_+\id^- )(\lopt\id)$ is $F$-$F$-special with 
\begin{align*} 
0 \leq b^{(g'_-\id^+-g'_+\id^- )(\lopt\id)}< \infty.
\end{align*}
\end{lemma}
\begin{proof} 
Fix $\lopt\in\R^d$ and for $\lambda\in\R^d$ consider 
$\xi^\lambda$ as in \eqref{eq:260925.1}.
By Lemma~\ref{L:260508London}, $\xi^\lambda$ is a V-function with 
\[
D\xi^\lambda(0)=g'(0)\zeta,
\qquad
D^2\xi^\lambda(0)
=g''(0)(\lopt^\top\zeta+\zeta^\top\lopt),
\]
where $\zeta=\lambda-\lopt$.
Observe that $\gfrak$ is concave and $\{0\}\subset\edom\gfrak\neq\emptyset$.
Furthermore, by monotone convergence, 
\begin{equation}\label{eq:260923Granada}
-\infty<\gfrak(\lambda)-\gfrak(\lopt)\leq\gfrak'(\lopt;\lambda-\lopt) = \lim_{K\uparrow\infty}b^{\xi^\lambda\wedge K},\qquad \lambda,\lopt\in\dom\gfrak.
\end{equation}

Suppose first that $\lopt$ maximizes $\gfrak$. Since $\gfrak(0)=0$, we have $\gfrak(\lopt)>-\infty$. There is
nothing to prove if $\gfrak(\lopt)=\infty$. Suppose therefore
that $\gfrak(\lopt)\in\R$. Then $\gfrak<\infty$. For $\lambda\in\edom\gfrak$, concavity
bounds $\xi^\lambda-g'(0)(\lambda-\lopt)\tr$ below by the
$F$-integrable function
\[
g(\lambda\id)-g(\lopt\id)-g'(0)(\lambda-\lopt)\tr.
\]
Moreover, optimality and \eqref{eq:260923Granada} imply that $\lim_{K\uparrow\infty}b^{\xi^\lambda\wedge K}$ is finite. Monotone convergence
therefore yields that $\xi^\lambda$ is ($F$-compatible and)
$F$-special, and hence
\[
-\infty<\gfrak'(\lopt;\lambda-\lopt)
=b^{\xi^\lambda}\leq0,
\qquad \lambda\in\edom\gfrak.
\]

Conversely, suppose that $\gfrak(\lopt)\in\R$
and $\xi^\lambda$ is $F$-special with
\[
    -\infty
    <
    b^{\xi^\lambda}
    \leq0,\qquad \lambda\in\edom\gfrak.
\]
Fix $x\in\R^d$ and consider the concave function $\psi:\R^d\to[-\infty,\infty]$ given by $\psi(\lambda) = g(\lambda x)$. 
By concavity of $g$, and therefore of $\psi$,
\[
g(\lambda x)-g(\lopt x)
\leq \psi'(\lopt;\lambda-\lopt) =
\xi^\lambda(x),
\qquad \lambda\in\R^d,x\in\R^d, \lopt x\in\dom g.
\]
Hence
\begin{align*} 
g(\lambda x)-g'(0)\lambda \tr(x)
\leq
g(\lopt x)-g'(0)\lopt \tr(x)
+
\xi^\lambda(x)-g'(0)\zeta \tr(x).
\end{align*}
Since $\lopt\in\dom\gfrak$ and $b^{\xi^\lambda}
    \leq0$ for all $\lambda\in\edom\gfrak$, by integrating the above we obtain 
 $\gfrak(\lambda)<\infty$ for any $\lambda\in\R^d$. Optimality now follows from \eqref{eq:260923Granada} with $\lim_{K\uparrow\infty}b^{\xi^\lambda\wedge K}=b^{\xi^\lambda}\leq0$.
\end{proof}

\section{Proof of Lemma~\ref{L:T}}\label{Appx:D}
\begin{proof}[Proof of Lemma~\ref{L:T}]
 Observe that $g''(0)<0$ yields $g'(0)>0$ by monotonicity of $g$. Furthermore, $k$ is twice Peano-differentiable at zero by
Lemma~\ref{L:260508London} and $g'(0)\id-g$ and $g-k$ are both twice Peano-differentiable at
zero with first derivative equal to $0$ and second derivative equal
to $-g''(0)>0$.

 Let
$\alpha_n=\beta_n=2^{-n}$ for $n\in\mathbb N$, and denote by
$D_n$ the predictable set on which $-2\alpha_n\in\dom g$ and, for
every $x\in\mathbb Q$ with $\abs{x}\leq\alpha_n$,
\[
    \beta_nx^2
    \leq
    g'(0)x-g(x)
    \leq
    \frac{1}{\beta_n}x^2,
    \qquad
    \beta_nx^2
    \leq
    g(x)-k(x).
\]
The sets $D_n$ are predictable and increase to
$\Omega\times[0,\infty)$, as $n$ tends to infinity, by the stated differentiability properties of the two functions at zero. 
Set now
$
    D'_n
    =
    D_n\setminus\bigcup_{\ell<n}D_\ell$.
Then
\[
    \alpha=\sum_{n\in\mathbb N}\alpha_n\indicator{D'_n},
    \qquad
    \beta=\sum_{n\in\mathbb N}\beta_n\indicator{D'_n},
    \qquad
    \gamma=\sum_{n\in\mathbb N}\frac{1}{\beta_n^2}\indicator{D'_n}
\]
are predictable and satisfy \eqref{eq:260519}. The inequalities initially
obtained for rational arguments extend to all arguments by continuity of
$g'(0)\id-g$, left-continuity of $g-k$ on $(0,\infty)$, and
right-continuity of $g-k$ on $(-\infty,0)$. 
 \end{proof}

\def\MR#1{\href{http://www.ams.org/mathscinet-getitem?mr=#1}{MR#1}}
\def\ARXIV#1{\href{https://arxiv.org/abs/#1}{arXiv:#1}}
\def\doilink#1{\href{https://doi.org/#1}{doi:#1}}

\end{document}